\documentclass{compositio}

\renewcommand{\contentsname}{Contents\\{\footnotesize\normalfont(A table
of contents should normally not be included)}}
\usepackage[english]{babel}
\usepackage{newlfont}
\usepackage{amssymb,amsfonts}
\usepackage{mathtools,braket,mathrsfs}
\usepackage[all]{xy}

\usepackage{caption}
\usepackage{subcaption}

\usepackage{booktabs}
\usepackage{halloweenmath}
\usepackage{stmaryrd}
\usepackage{bbm}
\usepackage{hyperref}
\usepackage{upgreek}

\usepackage{textcomp}

\newcommand{\plectic}[0]{\text{\textmarried}}

\usepackage[colorinlistoftodos,bordercolor=black,
backgroundcolor=orange!70!white,linecolor=black,]{todonotes}

\DeclareFontFamily{U}{wncy}{}
    \DeclareFontShape{U}{wncy}{m}{n}{<->wncyr10}{}
    \DeclareSymbolFont{mcy}{U}{wncy}{m}{n}
    \DeclareMathSymbol{\Sh}{\mathord}{mcy}{"58}

\newcommand{\bb}{\mathbb}
\newcommand{\mbf}{\mathbf}

\newcommand{\scr}{\mathscr}
\newcommand{\mrm}{\mathrm}
\newcommand*{\bfcdot}{\scalebox{0.6}{$\bullet$}}

\newcommand{\bs}{\boldsymbol}

\newcommand{\lmfdb}[5]{\href{https://www.lmfdb.org/EllipticCurve/2.2.#1.1/#2.#3/#4/#5}{{#2}.{#3}-{#4}{#5}}}

\theoremstyle{definition}
\newtheorem{theorem}{Theorem}[section]
\newtheorem{lemma}[theorem]{Lemma}
\newtheorem{proposition}[theorem]{Proposition}
\newtheorem{corollary}[theorem]{Corollary}
\newtheorem{definition}[theorem]{Definition}
\newtheorem{conjecture}[theorem]{Conjecture}
\newtheorem{problem}[theorem]{Problem}

\theoremstyle{remark}
\newtheorem{remark}[theorem]{Remark}

\def\Xint#1{\mathchoice
{\XXint\displaystyle\textstyle{#1}}%
{\XXint\textstyle\scriptstyle{#1}}%
{\XXint\scriptstyle\scriptscriptstyle{#1}}%
{\XXint\scriptscriptstyle\scriptscriptstyle{#1}}%
\!\int}
\def\XXint#1#2#3{{\setbox0=\hbox{$#1{#2#3}{\int}$ }
\vcenter{\hbox{$#2#3$ }}\kern-.585\wd0}}
\def\mint{\Xint\times}

\begin{document}

\title{Plectic $p$-adic invariants}

\author{Michele Fornea}
\email{mfornea@math.columbia.edu}
\address{Columbia University, New York, USA.}
\author{Xavier Guitart}
\email{xevi.guitart@gmail.com}
\address{Universitat de Barcelona, Barcelona, Catalonia.}
\author{Marc Masdeu}
\email{masdeu@mat.uab.cat}
\address{Universitat Aut\`onoma de Barcelona, Barcelona, Catalonia.}


\begin{abstract}
For modular elliptic curves over number fields of narrow class number one, and with multiplicative reduction at a collection of $p$-adic primes, we define new $p$-adic invariants.  Inspired by Nekov\'a$\check{\text{r}}$ and Scholl's plectic conjectures, we believe these invariants control the Mordell--Weil group of higher rank elliptic curves and we support our expectations with numerical experiments.
\end{abstract}

\maketitle


\section{Introduction}
In the late 1960s, Birch and Swinnerton-Dyer made a discovery that profoundly changed the study of the arithmetic of elliptic curves. Mordell had already shown in 1922 that the set of rational solutions $A(\bb{Q})$ of an elliptic curve is always a finitely generated abelian group. However, the rank of its free part, called the algebraic rank $r_\mrm{alg}(A/\bb{Q})$, was proving to be a subtle invariant difficult to compute. Birch and Swinnerton-Dyer's insight was to use local, and easy to compute, information about an elliptic curve to reconstruct its algebraic rank. Crucially informed by numerical calculations, they ended up formulating a momentous conjecture:

 for all but finitely many primes $\ell$, it is possible to reduce the Weierstrass equation of $A_{/\bb{Q}}$ modulo $\ell$ to obtain an elliptic curve $\bar{A}_{/\bb{F}_\ell}$ over a finite field. Every such curve is much simpler than the original one -- for instance, in the 1980s Schoof discovered a polynomial time algorithm\footnote{Birch and Swinnerton-Dyer performed their computations on CM elliptic curves whose number of points mod $\ell$ can be quickly computed using Hecke characters.} \cite{Schoof} that quickly computes the number $N_\ell(A):=\lvert\bar{A}(\bb{F}_\ell)\rvert$ of points modulo $\ell$. As there are natural reduction maps $A(\bb{Q})\to \bar{A}(\bb{F}_\ell)$, one could heuristically expect a large algebraic rank to force the sets $\bar{A}(\bb{F}_\ell)$ to also be large on average. Birch and Swinnerton-Dyer turned this heuristic into a quantitative mathematical statement and successfully tested it  on a computer. They noticed that a normalized product of $N_\ell(A)$'s grows as the $r_\mrm{alg}(A/\bb{Q})$-th power of the logarithm function
\begin{equation}\label{BSD}
\prod_{\ell\le X}\frac{N_\ell(A)}{\ell}\quad\overset{?}{\mbox{\large$\sim$}}_{\mbox{\tiny $+\infty$}}\quad\big(\log X\big)^{r_\mrm{alg}(A/\bb{Q})}.
\end{equation}
The expectation that every elliptic curve over a number field satisfies an appropriate generalization of \eqref{BSD} became known as the BSD conjecture\footnote{Goldfeld proved that the asymptotic conjecture of Birch and Swinnerton-Dyer implies the modern formulation of the BSD conjecture in terms of order of vanishing of $L$-functions  \cite{GoldfeldBSD}.}. The intrinsic appeal of the problem has inspired the research of many mathematicians who unearthed surprising connections with different parts of mathematics. 
Among the landmarks in the field one finds (in publishing order) the works of Coates--Wiles \cite{CoatesWiles}, Gross--Zagier \cite{GZformula}, Kolyvagin \cite{Koly}, Zhang \cite{Heights}, Skinner--Urban \cite{IwasawaSU}, Darmon--Rotger \cite{DR2} and Skinner \cite{Skinner}. 
It is interesting to note that all these works are limited to elliptic curves of small rank, the ultimate reason for this arguably being that the only known general approach to prove that the algebraic rank is at least $r$ is to produce $r$ linearly independent points. Therefore, the need becomes apparent for a \emph{systematic} strategy to construct linear independent points on elliptic curves. In order to formulate such task precisely, we recall that when studying an elliptic curve $A$ over a number field $F$, experience has shown that is beneficial to consider an auxiliary quadratic extension $E/F$.
\begin{problem}\label{problem}
Construct an element $P\in \wedge^r A(E)$ such that 
\[
P\ \ \text{non-torsion}\quad\iff\quad r_\mrm{alg}(A/E)=r.
\]
\end{problem}
\noindent For $r=1$, the constructions of Heegner and Stark--Heegner points give a satisfactory answer to Problem \ref{problem} -- with a caveat. While Heegner points are defined for quadratic CM extensions and are known to be algebraic by the theory of complex multiplication, their generalization to arbitrary quadratic extensions $E/F$, called Stark--Heegner points, are only conjecturally algebraic in general. Stark--Heegner points were first defined for real quadratic fields in \cite{IntegrationDarmon} and later generalized to arbitrary extensions in (\cite{Greenberg}, \cite{ArbitraryDarmon}, \cite{AutomorphicDarmon}). While their definition is inherently local and relies on analytic methods, the theoretical and numerical evidence is so overwhelming that it is commonly accepted that Heegner and Stark--Heegner points completely control the Mordell--Weil group of elliptic curves of rank one.

In this paper we propose a generalization of the $p$-adic construction of Stark--Heegner points inspired by Nekov\'a$\check{\text{r}}$ and Scholl's plectic conjectures (\cite{PlecticNS}, \cite{NekRubinfest}). The outcome is a conjectural answer to Problem \ref{problem} for arbitrary $r\ge2$, under some assumptions on the elliptic curve $A_{/F}$ and the quadratic extension $E/F$. Throughout the text we enforce some \emph{simplifying hypotheses} to tailor the exposition to numerical verication: we only consider non-CM quadratic extensions $E/F$ where both fields have narrow class number one. We refer to \cite{plecticHeegner} for the construction of plectic invariants in the general case, and for the definition of certain refined invariants that may be called \emph{plectic Stark--Heegner points}. Finally, we direct the adventurous reader to  \cite{PlecticJacobians} for the speculative  framework of \emph{plectic Jacobians} recasting the construction of plectic Heegner points in geometric terms.  

We continue the introduction by outlining the construction of plectic $p$-adic invariants and by formulating precise conjectures regarding their significance for the arithmetic of higher rank elliptic curves. Finally, we describe our numerical experiments.

\subsection{Overview} 
 Let $F$ be a number field with $t$ real places, $s$ complex places, and narrow class number one. We consider $A_{/F}$ a modular elliptic curve of conductor $\frak{f}_A$, and $E/F$ a non-CM quadratic extension of narrow class number one where $\frak{f}_A$ is unramified. 
We aim to define an invariant whose non-triviality implies that the Mordell--Weil group $A(E)$ has rank $r$. To this end, we fix a rational prime $p$ unramified in $F$, and a set  $S=\{\frak{p}_1,\dots,\frak{p}_r\}$ of $r$ distinct $p$-adic $\cal{O}_F$-prime ideals all inert in $E$. Denote by $\widehat{E}_{\frak{p}}^\times$ the free part of the $p$-adic completion of $E_\frak{p}^\times$; the \emph{plectic $p$-adic invariant} associated to a triple $(A_{/F},E,S)$,  satisfying certain hypotheses described below, is an element  in the tensor product $\widehat{E}_{S,\otimes}^\times:=\otimes_{\frak{p}\in S}\widehat{E}_{\frak{p}}^\times$ of $\bb{Z}_p$-modules
\[
\mrm{Q}_{A} \in \widehat{E}_{S,\otimes}^\times.
\]
The construction follows closely the strategy of \cite{Greenberg}, \cite{ArbitraryDarmon}. However, while those approaches are modeled on the $p$-adic uniformization of Shimura curves, ours is inspired by the $p$-adic uniformization of higher dimensional quaternionic Shimura varieties by products of $p$-adic upper-half planes \cite{Varshavsky}. 
 
\noindent We continue by explaining the conditions that we need to impose on the triple $(A_{/F},E,S)$.  Suppose that every $\frak{p}\in S$ divides \emph{exactly} the conductor $\frak{f}_A$, and that we can write
\[
\frak{f}_A=p_S\cdot\frak{n}^{\mbox{\tiny $+$}}\cdot\frak{n}^{\mbox{\tiny $-$}},
\]
where $p_S=\prod_{\frak{p}\in S}\frak{p}$ and $\frak{n}^{\mbox{\tiny $+$}}$ is the product of \emph{all} prime divisors of $\frak{f}_A$ that split in $E$. Further, we require $\frak{n}^{\mbox{\tiny $-$}}$ to be square-free and denote by $\omega(\frak{n}^{\mbox{\tiny $-$}})$ the number of its prime factors. If we denote by $\infty_1,\dots,\infty_{t}$ the real places of $F$ ordered such that the first $n$ are precisely those that split in $E$, then the root number of $A_{/E}$ is conjecturally computed by
\[
\varepsilon(A/E)=(-1)^{r+\omega(\frak{n}^{\mbox{\tiny $-$}})+(t-n)}.
\]
The parity conjecture suggests us to impose the congruence condition $\omega(\frak{n}^{\mbox{\tiny $-$}})\equiv (t-n)\pmod{2}$.
Under these hypotheses there is a unique quaternion algebra $B/F$ ramified precisely at 
\[
\{ \frak{q}\mid \frak{n}^{\mbox{\tiny $-$}}\}\cup\{\infty_{n+1},\dots,\infty_{t}\}
\]
and admitting an embedding $E\hookrightarrow B$. Moreover,  $B/F$ is \emph{not} totally definite because  the quadratic extension $E/F$ is \emph{not} CM.  The construction of the invariants is naturally divided into three steps.

\subsubsection{Construction.}
Let $F_S$ and $B_S$ denote the $S$-adic completions of $F$ and $B$ respectively. The choice of an isomorphism $\iota\colon B_S^\times\overset{\sim}{\to}\mrm{GL}_2(F_S)$ and the Jacquet--Langlands correspondence allow us to transform the cohomological eigenform of weight $2$ associated to $A_{/F}$ by modularity into a measure-valued eigenclass
	\[
	\mbf{c}_A\in \mrm{H}^{n+s}\big(\Gamma,\scr{M}_\plectic(\bb{P}^1(F_S),\bb{Z})\big)_\pi.
	\]
	The cohomology is computed with respect to the $S$-arithmetic subgroup $ \Gamma:=R_1^\times/\{\pm1\}$ of $B^\times/F^\times$ arising from an $\cal{O}_{F}[\mbox{\small $\tiny S^{-1}$}]$-Eichler order $R$ of level $\frak{n}^{\mbox{\tiny $+$}}$, while the coefficients $\scr{M}_\plectic(\bb{P}^1(F_S),\bb{Z})$  are $\bb{Z}$-valued measures $\mu$ on $\bb{P}^1(F_S)=\prod_{\frak{p}\in S}\bb{P}^1(F_\frak{p})$ satisfying the following property:
the value $\mu(U)$ equals zero if the open compact subset $U\subseteq \bb{P}^1(F_S)$ can be written as
	\[
	U=\bb{P}^1(F_\frak{p})\times V\quad\text{for}\quad V\subseteq\bb{P}^1(F_{S\setminus\{\frak{p}\}})\quad \text{compact open}.
	\]
 The second step comprises an extension of the theory of $p$-adic multiplicative integrals. The characterizing property of the measures in $\scr{M}_\plectic(\bb{P}^1(F_S),\bb{Z})$ suggests that it should be possible to define a meaningful integration pairing with a certain subgroup of zero-cycles on $\cal{H}_S=\prod_{\frak{p}\in S}\cal{H}_\frak{p}$ where $\cal{H}_\frak{p}=\bb{P}^1(E_\frak{p})\setminus \bb{P}^1(F_\frak{p})$. Indeed, any measure $\mu\in \scr{M}_\plectic(\bb{P}^1(F_S),\bb{Z})$ can be integrated against a zero-cycle of the form $C=\otimes_{\frak{p}\in S}([x_\frak{p}]-[y_\frak{p}])$ by computing a limit of Riemann products
\[
\mint_{\bb{P}^1(F_S)}\ \bigotimes_{\frak{p}\in S}\bigg(\frac{t_\frak{p}-x_\frak{p}}{t_\frak{p}-y_\frak{p}}\bigg)\ \mrm{d}\mu(t)\ \in\ \widehat{E}_{S,\otimes}^\times.
\]
Therefore, it makes sense to define the group $\bb{Z}_\plectic[\cal{H}_S]$ of plectic zero-cycles on $\cal{H}_S$  as the tensor product $\otimes_{\frak{p}\in S}\bb{Z}[\cal{H}_\frak{p}]^0$ of zero-cycles of degree zero. We obtain a $\mrm{PGL}_2(F_S)$-equivariant pairing
 \[ \mint\colon\scr{M}_\plectic\big(\bb{P}^1(F_S),\bb{Z}\big)\times \bb{Z}_\plectic[\cal{H}_S]\longrightarrow \widehat{E}_{S,\otimes}^\times
 \]
giving rise to a cap product pairing $\cap\colon \mrm{H}^{n+s}\big(\Gamma, \scr{M}_\plectic(\bb{P}^1(F_S),\bb{Z})\big)\times \mrm{H}_{n+s}\big(\Gamma,\bb{Z}_\plectic[\cal{H}_S]\big) \longrightarrow \widehat{E}_{S,\otimes}^\times$.

\noindent The final step consists in defining a homology class $\Delta_E\in \mrm{H}_{n+s}\big(\Gamma,\bb{Z}_\plectic[\cal{H}_S]\big)$ associated to the quadratic extension $E/F$. As $E$ has narrow class number one, the class $\Delta_E$ is essentially unique.
\begin{definition}
	The \emph{plectic $p$-adic invariant} associated to the triple $(A_{/F},E,S)$ is given by
		\[
		\mrm{Q}_{A}:=\mbf{c}_A\cap\Delta_E\ \in\ \widehat{E}_{S,\otimes}^\times.
		\]
\end{definition}
\noindent Plectic invariants encode interesting arithmetic information: they appear as values of higher derivatives of anticyclotomic $p$-adic $L$-functions (\cite{plecticHeegner}, Theorem A) and they can be unconditionally defined in great generality. Furthermore, we can use  Tate's uniformizations $E_\frak{p}^{\times}\to A(E_\frak{p})$ at every prime $\frak{p}\in S$ to conjecturally relate them to global points of the elliptic curve. Let $\widehat{A}(E_\frak{p})$ denote the free part of the $p$-adic completion of $A(E_\frak{p})$ and let $\widehat{A}(E_{S}):=\bigotimes_{\frak{p}\in S}\widehat{A}(E_{\frak{p}})$ denote the tensor product of $\bb{Z}_p$-modules. Then, the uniformization $\phi_{\mbox{\tiny $\mrm{Tate}$}} \colon  \widehat{E}_{S,\otimes}^{\times}\rightarrow \widehat{A}(E_{S})$ allows us to consider the \emph{plectic point}
\[
\phi_{\mbox{\tiny $\mrm{Tate}$}}(\mrm{Q}_{A})\ \in\ \widehat{A}(E_{S})
\]
whose expected algebraicity and relevance for the Mordell--Weil group $A(E)$ are presented next.

\subsection{Conjectures}\label{conjectures}
Recall that both fields, $F$ and $E$, have narrow class number one, that we fixed a rational prime $p$ and a set $S=\{\frak{p}_1,\dots,\frak{p}_r\}$ of $p$-adic $\cal{O}_F$-prime ideals inert in $E$. We also made assumptions on the conductor $\frak{f}_A$ of the modular elliptic curve $A_{/F}$
so that we can expect the congruence $r_\mrm{alg}(A/E)\equiv_2r$ to be satisfied.  Now, for every $\frak{p}\in S$ we fix $\iota_\frak{p}\colon E\hookrightarrow E_\frak{p}$ an embedding of $E$ in its completion $E_\frak{p}$, and define a homomorphism $\det\colon \wedge^{r} A(E)\to \widehat{A}(E_S)$
by setting
\[
\det\big(P_1\wedge\dots\wedge P_r\big)=\det \begin{pmatrix}
		\iota_{\frak{p}_1}(P_1)&\dots& \iota_{\frak{p}_r}(P_1)\\
	&\dots&\\
	\iota_{\frak{p}_{1}}(P_r)&\dots& \iota_{\frak{p}_r}(P_r)
	\end{pmatrix},
\]
where we take tensor products whenever the formula computing the determinant would require multiplication of entries of the matrix. Interestingly, we can further pin down the position of $\phi_{\mbox{\tiny $\mrm{Tate}$}}(\mrm{Q}_{A})$ inside $\widehat{A}(E_S)$ using as clue the definition of the class $\Delta_E$. Partition $S=S^+\cup S^-$ by declaring that the subset $S^+\subseteq S$ contains \emph{all} the primes in $S$ of split multiplicative reduction for $A_{/F}$. Further, consider the map
\[
\pi_S\colon\widehat{A}(E_S)\longrightarrow \widehat{A}(E_S),\qquad \pi_S(R)=\bigg(\prod_{\frak{p}\in S^+}(1-\sigma_\frak{p}^*) \prod_{\frak{p}\in S^-}(1+\sigma_\frak{p}^*)\bigg)R,
\]
where the ``partial Frobenius'' $\sigma_\frak{p}\in\text{Gal}(E_\frak{p}/F_\frak{p})$ acts on $\widehat{A}(E_S)$ through its natural action on the $\frak{p}$-th component. Then, after defining 
\[
\mrm{det}_S:=\pi_S\circ\det,
\]
the meaning of algebraicity for the element $\phi_{\mbox{\tiny $\mrm{Tate}$}}(\mrm{Q}_{A})$ takes the following form.
\begin{conjecture}(Algebraicity)\label{algebraicity}
If $r_\mrm{alg}(A/E)\ge r$, there exists an element $w_A\in \wedge^r A(E)$ s.t.
	\[
	\phi_{\mbox{\tiny $\mrm{Tate}$}}(\mrm{Q}_{A})=\mrm{det}_S(w_A).
	\]	
\end{conjecture}
\begin{remark}
	The homomorphism $\det_S$ may have a non-trivial kernel even when $r_\mrm{alg}(A/E)= r$.
\end{remark}

\noindent On its own, Conjecture \ref{algebraicity} does not reveal enough about the intimate relation between the plectic point $\phi_{\mbox{\tiny $\mrm{Tate}$}}(\mrm{Q}_{A})$ and the global arithmetic of the elliptic curve $A_{/E}$. Thus, we couple it with another conjecture of Kolyvagin-type described below.
 Let $A^+=A$, denote by $A^-$ the quadratic twist of $A_{/F}$ with respect to the quadratic extension $E/F$, and define
\[
\varrho_{A}(S):=\text{max}\big\{r_\mrm{alg}(A^\pm/F)+\lvert S^\pm\rvert\big\}.
\]
The anticyclotomic $p$-adic $L$-function $\scr{L}_S(A/E)$ attached to the triple $(A_{/F},E,S)$ in \cite{plecticHeegner} vanishes to order at least $r$ at the trivial character, and the value of its $r$-th derivative computes the plectic $p$-adic invariant $\mrm{Q}_A$. Moreover, when $r_\mrm{alg}(A/E)\ge r$, the theory of derived $p$-adic heights \cite{derivedheights} suggests that $\varrho_{A}(S)\ge r$ is a lower bound for the exact order of vanishing of the $p$-adic $L$-function. It is then natural to formulate the following conjecture.

\begin{conjecture}\label{Kolyvagin}
Suppose that $r_\mrm{alg}(A/E)\ge r$, then 
	\[
	\phi_{\mbox{\tiny $\mrm{Tate}$}}(\mrm{Q}_{A})\not=0\quad\implies\quad r_\mrm{alg}(A/E)=r\quad\&\quad \varrho_{A}(S)=r.
	\]	
	If the $L$-function $L(A/F,s)$ is primitive, then the converse implication also holds.
\end{conjecture}

\noindent 	When the algebraic rank $r_\mrm{alg}(A/E)$ is strictly smaller than $r$ we cannot yet guess what arithmetic information is contained in the plectic $p$-adic invariant. Nevertheless,  we expect that -- assuming $r_\mrm{alg}(A/E)<r$ -- the plectic point should be non-zero whenever the parity of $r=\lvert S\rvert$ matches  the  parity of the order of vanishing of $\scr{L}_S(A/E)$ (see \cite{AnticyclotomicBG}, Corollary 5.7).
	
\begin{conjecture}\label{lowerrank}
Suppose $r_\mrm{alg}(A/E)<r$, then
		\[
		\phi_{\mbox{\tiny $\mrm{Tate}$}}(\mrm{Q}_{A})\not=0\quad\iff\quad (-1)^r=\varepsilon(A/F)\cdot \varepsilon_{S},
		\] 
		where $\varepsilon_S$ is the product of local root numbers of $A_{/F}$ at the primes in $S$. 
\end{conjecture}

\subsection{Numerical evidence}\label{Numerical evidence}
We are most excited in supporting our conjectures with computational experiments. We choose to perform our calculation in the setting of $F$ a real quadratic number field and $E/F$ an almost totally real (ATR) extension in order to work in cohomological degree one, and to have a supply of elliptic curves of positive rank and relatively small conductor. We consider the two smallest real quadratic fields of narrow class number one, where the prime $3$ splits:
\[
\bb{Q}(\sqrt{13}),\  \bb{Q}(\sqrt{37}).
\] 
For each base field we then test our conjectures on a few isogeny classes of elliptic curves, whose conductor satisfies the requirement for the definition of plectic $3$-adic invariants over ATR extensions.
The code is implemented in Sage \cite{sagemath} (using some Magma  functions \cite{magma} at certain steps), and is available at \href{https://github.com/mmasdeu/darmonpoints/blob/master/darmonpoints/plectic.py}{this link}.\footnote{
\url{https://github.com/mmasdeu/darmonpoints/blob/master/darmonpoints/plectic.py}}
Moreover, it makes use of some functionalities of the Darmon Points package, previously developed by Masdeu and available at the same Github repository.
	  The main difficulty in computing plectic invariants consists in explicitly describing the class 
	  \[
	  \mbf{c}_A\in \mrm{H}^1(\Gamma,\scr{M}_\plectic(\bb{P}^1(F_S),\bb{Z}))_\pi.
	  \]
	In Section \ref{algorithm} we describe the algorithm that computes approximations $\mbf{c}^m_A$ of $\mbf{c}_A$ for integers $m\ge0$. Note that we do not compute directly the plectic invariant, but rather its image under a logarithm map:  let $q_{\frak{p}}\in F_{\frak{p}}^\times$ denote the period of the Tate curve $A_{/F_{\frak{p}}}$ for $\frak{p}\in S$, and let $E_{S,\otimes}$ denote the tensor product $\otimes_{\frak{p}\in S}E_{\frak{p}}$ of $\bb{Q}_p$-vector spaces; then there is a homomorphism $ \log_{S}\colon\widehat{E}^\times_{S,\otimes}\rightarrow E_{S,\otimes}$ -- factoring through Tate's uniformization $\log_{S}=\log_A\circ\ \phi_{\mbox{\tiny $\mrm{Tate}$}}$ -- such that 
	 \[	  \log_{S}(\mrm{Q}_{A})=\int_{\bb{P}^1(F_S)}\bigotimes_{\frak{p}\in S} \log_{q_{\frak{p}}} \left(\frac{t_{\frak{p}}-\tau_{\frak{p}}}{t_{\frak{p}}-\bar\tau_{\frak{p}}} \right)\ \mrm{d}\mu_A(t).
	 \]
	 As the induced morphism $\log_A:\widehat{A}(E_{S})\to E_{S,\otimes}$ is injective, we do not lose information about $\phi_{\mbox{\tiny $\mrm{Tate}$}}(\mrm{Q}_{A})$ in computing the additive $p$-adic integral for $\log_{S}(\mrm{Q}_{A})$. 
Concretely, we numerically calculate Riemann sums of the form
      \[
        \sum_{U_e\in \scr{U}_m} \bigotimes_{\frak{p}\in S} \log_{q_{\frak{p}}} \left(\frac{t_{e,\frak{p}}-\tau_{\frak{p}}}{t_{e,\frak{p}}-\bar\tau_{\frak{p}}} \right)\cdot \mbf{c}_A^m(e),
      \]
      where $\scr{U}_m$ is certain covering of $\mathbb{P}^1(F_S)$ and $t_e\in U_e$.  We conclude the introduction with an example of our results, and refer to Section \ref{Lastsection} for the rest of our data. 
	  
\begin{remark}
	The algorithm based on Riemann sums is exponential in time; this explains the limited number of digits we can compute. In future work, we plan to obtain a polynomial time algorithm by adapting the methods of \cite{GM-Oc} to the setting of this paper.
\end{remark}	  
	  
\noindent Consider the elliptic curve with LMFDB \cite{lmfdb} label \lmfdb{37}{63} 2d1 defined over $F=\mathbb{Q}(\sqrt{37})$ whose conductor $\frak{f}_A=\frak{p}_1\cdot \frak{p}_2\cdot \frak{q}$ has norm $63=3\cdot 3\cdot7$. It has a Weierstrass model
  \[
    A_{/F} \colon\ y^2 + x y + y = x^{3} + w x^{2} + \left(w + 1\right) x + 2, \qquad w = \frac{1+\sqrt{37}}2.
  \]
The elliptic curve has rank $2$ over the ATR extension $E=F\big(\sqrt{\beta}\big)$, for $\beta=62-21w$, where all prime divisors of $\frak{f}_A$ are inert. Moreover, by setting $S=\{\frak{p}_1,\frak{p}_2\}$ we find that $\varrho_A(S)=2$ because $A_{/F}$ has rank $1$ over $F$, and it has both a $3$-adic prime of split and non-split multiplicative reduction. Therefore, Conjectures \ref{algebraicity} and \ref{Kolyvagin} imply that the plectic point is non-zero and explicitly related to a generator of $\wedge^2A(E)$. We use \texttt{Magma} to compute the generators for $A(E)$
  \[
  P_1=\big(3-w,\ 4-w\big),\qquad  P_2=\left(8-\frac{25}{9} w,\ \left(-\frac{23}{27}w + \frac{17}6\right)\sqrt{\beta} + \frac{25}{18}w - \frac 92\right).
  \]
 Let $\bb{Q}_9=\bb{Q}_3(\sqrt{-1})$ denote the unramified quadratic extension of $\bb{Q}_3$ and set $P_S=\det_S(P_1\wedge P_2)$. Using \texttt{Sage}, we compute the following approximation
  \[
   \log_A(P_S) =
    \big(2 \cdot 3^{2} + 3^{6} + 2 \cdot 3^{7} + 3^{9} + O(3^{10})\big)\cdot \big(\sqrt{-1}\otimes\sqrt{-1}\big)
  \]
as an element of $\mathbb{Q}_{9}\otimes_{\bb{Q}_3} \mathbb{Q}_{9}$. For the computation of the plectic invariant we used a 72-CPU cluster with $500\text{GB}$ RAM at the University of Warwick. To approximate the cohomology class $\mbf{c}_A$ modulo $3^7$ we solve a linear system of $12,740,008$ equations in $19,114,384$ unknowns. It takes $\sim\hspace{-0.5mm}60$ hours, using 16 CPUs, to build the system, and $\sim \hspace{-1mm}2$ hours (non-parallel) to solve it. Finally, the integration step takes $\sim\hspace{-1mm} 10$ hours with 64 CPUs. In total, the computation used $\sim\hspace{-0.8mm} 300\text{GB}$ of RAM memory. Our algorithms compute the quantity
\[
\log_S(\mrm{Q}_A) = \big(2 \cdot 3^{2} + 3^{6} + O(3^7)\big)\cdot \big(\sqrt{-1}\otimes\sqrt{-1}\big),
\] 
which matches $ \log_A(P_S)$ modulo $3^7$.

\begin{acknowledgements}
We would like to express our deep gratitute to Lennart Gehrmann for sharing the exploration of the plectic world, to Matteo Tamiozzo for explaining Nekov\'a$\check{\text{r}}$ and Scholl's conjectures to us, and to the Mathematics Department at the University of Warwick whose computer facilities carried out our more demanding computations. Moreover, we thank Dante Bonolis, Henri Darmon, David Lilienfeldt, Jan Nekov\'a$\check{\text{r}}$, Kartik Prasanna, Tony Scholl, Nicolas Simard, Ari Shnidman, Christopher Skinner and Jan Vonk for their support and for many enriching conversations.  
 
 \noindent The article was completed while Fornea was a Simons Junior Fellow and Guitart was partially funded
by project PID2019-107297GB-I00. Moreover, this work has received funding from the European Research Council (ERC) under the European
Union’s Horizon 2020 research and innovation programme (grant agreement No 682152).
\end{acknowledgements}

\section{ Review of group cohomology}\label{sec: review of group cohomology}
We keep the hypotheses of the introduction: $F$ is a number field with $t$ real places, $s$ complex places, and narrow class number one. We let $n\le t$ be a non-negative integer and consider a square-free $\cal{O}_F$-ideal $\frak{n}^{\mbox{\tiny $-$}}$ whose number of prime factors $\omega(\frak{n}^{\mbox{\tiny $-$}})$ satisfies $\omega(\frak{n}^{\mbox{\tiny $-$}})\equiv_2(t-n)$. 

\noindent Let $B/F$ be a quaternion algebra of discriminant  $\frak{n}^{\mbox{\tiny $-$}}$ and whose ramified archimedean places are  the real places $\infty_{n+1},\dots,\infty_t$ of $F$. 
For each $\cal{O}_F$-ideal $\frak{c}$ prime to $\frak{n}^{\mbox{\tiny $-$}}$ we choose an Eichler $\cal{O}_F$-order $R_{\frak{c}}$ in $B$ of level $\frak{c}$ such that $\frak{c}_1\mid\frak{c}_2$ implies $R_{\frak{c}_2}\subseteq R_{\frak{c}_1}$. For a subgroup $G$ of $B^\times$ we set 
\[\begin{split}
G_1&:=\{\alpha\in G:\ \mrm{nrd}(\alpha)=1\},\\
G_+&:=\{\alpha\in G:\ \infty_i(\mrm{nrd}(\alpha))>0\ \ \forall\ i=1,\dots,n\}
\end{split}\]
where $\mrm{nrd}:B^\times\to F^\times$ is the reduced norm, and define the arithmetic subgroups of $B^\times/F^\times$  by
\[
\Gamma_{\frak{c}}:=(R_{\frak{c}}^\times)_1/\{\pm1\}.
\]
As we will describe in the next section, the cohomology of these arithmetic subgroups is endowed with an action of certain global units. A first indication of this fact is that the quotient of
\[
U_B=\{u\in\cal{O}^\times_F:\ \infty_i(u)>0\ \ \forall\ i=n+1,\dots, t\},
\]
 by the group $U_+$  of totally positive units of $F$, fits in a short exact sequence (\cite{Greenberg}, Section 2)
\begin{equation}\label{ses1}
	\xymatrix{
1\ar[r]& \Gamma_{\frak{c}}\ar[r]& R_{\frak{c}}^\times/\cal{O}_F^\times\ar[r]^-{\mrm{nrd}}&U_B/U_+\ar[r]&1.
}\end{equation}
Analogous statements hold for the corresponding $S$-arithmetic sugroups: we fix a rational prime $p$, a set $S=\{\frak{p}_1,\dots,\frak{p}_r\}$ of $r$ distinct $p$-adic $\cal{O}_F$-prime ideals, and an ideal $\frak{n}^{\mbox{\tiny $+$}}$ prime to $\frak{n}^{\mbox{\tiny $-$}}$. If we denote by $\cal{O}_F[\mbox{\small $S^{-1}$}]$ the ring of $S$-integers, then we can consider the $\cal{O}_F[\mbox{\small $S^{-1}$}]$-Eichler order of $B$
\[
R:=R_{\frak{n}^{\mbox{\tiny $+$}}}\otimes_{\cal{O}_F}\cal{O}_F[\mbox{\small $S^{-1}$}]
\]
 of level $\frak{n}^{\mbox{\tiny $+$}}$ whose associated $S$-arithmetic group is $\Gamma=R_1^\times/\{\pm1\}$. Moreover, the reduced norm induces a short exact sequence
 \begin{equation}\label{ses2}
 \xymatrix{
 1\ar[r]& \Gamma\ar[r]& R_\mrm{ev}^\times/\cal{O}_F[\mbox{\small $S^{-1}$}]^\times\ar[r]^-{\mrm{nrd}}&U_B/U_+\ar[r]&1
 }\end{equation}
where $R^\times_\mrm{ev}=\{\alpha\in R^\times:\ \mrm{ord}_\frak{p}\circ\mrm{nrd}(\alpha)\equiv_20\quad \forall\ \frak{p}\in S\}$ is the group of invertible elements of $R$ whose reduced norm has even valuation at all the primes in $S$.

\subsection{Hecke operators}
In this subsection we let $G$ be either $\Gamma$ or $\Gamma_{\frak{c}}$ for some $\cal{O}_F$-ideal $\frak{c}$ prime to $\frak{n}^{\mbox{\tiny $-$}}$.
For any $\cal{O}_F$-prime ideal $\frak{q}\nmid p_S\cdot\frak{n}^{\mbox{\tiny $-$}}$ we choose a totally positive generator $\varpi_\frak{q}$, then there is an element $\gamma_\frak{q}\in B^\times$ of reduced norm $\varpi_\frak{q}$ such that the intersection of $G$ with $G^{\gamma_\frak{q}}=\gamma_\frak{q} G\gamma_\frak{q}^{-1}$ has finite index in both groups. For any $G$-module $M$, the Hecke operator $T_\frak{q}$ is the linear endomorphism on cohomology defined by the composition
\[\xymatrix{
\mrm{H}^i(G,M)\ar@{.>}[rr]^-{T_\frak{q}}\ar[d]_-{\mrm{res}}&& \mrm{H}^i(G,M)\\
\mrm{H}^i\big(G\cap G^{\gamma_\frak{q}},M\big)\ar[rr]^-{\mrm{conj}}&& \mrm{H}^i\big(G^{\gamma_\frak{q}^{-1}}\cap G,M\big).\ar[u]^-{\mrm{cores}}&
}\]
Furthermore, when $M$ is either an $(R_{\frak{c}}^\times/\cal{O}_F^\times)$-module or an $(R_\mrm{ev}^\times/\cal{O}_{F}[\mbox{\small $S^{-1}$}]^\times)$-module the short exact sequences ($\ref{ses1}$), ($\ref{ses2}$) endow the cohomology
  $\mrm{H}^{\bfcdot}(\Gamma, M)$ and $\mrm{H}^{\bfcdot}(\Gamma_{\frak{c}}, M)$ with an action of $U_B/U_+$. 
 \begin{definition}
Let $\delta_{i,j}$ be Kronecker's delta function.   For every index $i=1,\dots, n$, we fix a global unit $\varepsilon_i\in U_B$ satisfying
\[
(-1)^{\delta_{i,j}}\cdot\infty_j(\varepsilon_i)>0\quad \forall\ j=1,\dots, n.
\]
 We denote by $T_{i}$ the involution  determined by the coset $\varepsilon_iU_+$ acting via ($\ref{ses1}$), ($\ref{ses2}$) on cohomology.
\end{definition}

\subsection{Cohomology of arithmetic subgroups}
We denote by $\cal{H}$ the Poincar\'e upper half-plane endowed with the action of $\mrm{PSL}_2(\bb{R})$ by Moebius transformation, and we consider the half-space model of the real hyperbolic $3$-space
\[
\bb{H}=\{(x,y)\in\bb{C}\times\bb{R}\lvert\ y>0\}
\]
 endowed with the $\mrm{PSL}_2(\bb{C})$-action explicitly given by
 \[
\begin{pmatrix}
	 a&b\\ c&d
	 \end{pmatrix}\cdot(x,y)=\left(\frac{(ax+b)\overline{(cx+d)}+a\bar{c}y^2}{\lvert cx+d\rvert^2+\lvert cy\rvert^2}, \frac{\lvert ad-bc\rvert y}{\lvert cx+d\rvert^2+\lvert cy\rvert^2}\right).
 \]
By fixing isomorphisms for the archimedean completions
\[\begin{split}
&\iota_{i}:B\otimes_{F,\infty_i}\bb{R}\overset{\sim}{\longrightarrow}\mrm{M}_2(\bb{R})\quad\ \ \ \ \text{for}\  i=1,\dots,n,\\
 &\iota_{j}:B\otimes_{F,\infty_j}\bb{C}\overset{\sim}{\longrightarrow}\mrm{M}_2(\bb{C})\qquad \text{for}\ j=t+1,\dots,t+s,
\end{split}\]
the arithmetic subgroup $\Gamma_\frak{c}$ act on the manifold $\cal{H}^n\times\bb{H}^s$. The quotient is a Riemannian orbifold
\[
X_\frak{c}:=\Gamma_\frak{c}\backslash\big(\cal{H}^n\times\bb{H}^s\big)
\]
 whose rational cohomology computes group cohomology $\mrm{H}^{\bfcdot}(X_\frak{c},\bb{Q})\cong\mrm{H}^{\bfcdot}(\Gamma_\frak{c},\bb{Q})$.
We recall some useful facts about the $\bb{C}$-valued cohomology of $X_\frak{c}$ taken from (\cite{ArbitraryDarmon}, Section 2). The cohomology $\mrm{H}^{\bfcdot}(X_\frak{c},\bb{C})$ contains three subspaces that are preserved by the action of Hecke operators:
\begin{itemize}
	\item [$\bfcdot$] the \emph{universal} subspace $\mrm{H}^{\bfcdot}_\mrm{uni}(X_\frak{c},\bb{C})$ is the image of $\Gamma_\frak{c}$-invariant differential forms on $\cal{H}^n\times\bb{H}^s$;
	\item [$\bfcdot$] the \emph{infinity} subspace $\mrm{H}^{\bfcdot}_\mrm{inf}(X_\frak{c},\bb{C})$ is the image of the cohomology of the boundary of the Borel-Serre compactification of $X_\frak{c}$;
	\item [$\bfcdot$] the \emph{cuspidal} subspace $\mrm{H}^{\bfcdot}_\mrm{cusp}(X_\frak{c},\bb{C})$ is generated  by cohomological cuspforms of weight $2$ and their translates by the archimedean involutions.	
\end{itemize}
The Hecke operators act on the non-cuspidal parts $\mrm{H}^{\bfcdot}_\mrm{uni}(X_\frak{c},\bb{C})$, $\mrm{H}^{\bfcdot}_\mrm{inf}(X_\frak{c},\bb{C})$ through multiplication by their degree, i.e.
\[
 \mrm{deg}(T_{i})=+1,\quad\mrm{deg}(T_\frak{q})=\begin{cases}
	\mrm{N}_{F/\bb{Q}}(\frak{q})+1& \text{if}\ \ \frak{q}\nmid\frak{c}\\
	\mrm{N}_{F/\bb{Q}}(\frak{q})& \text{if}\ \ \frak{q}\mid\frak{c}
\end{cases}.
\]
Furthermore, the cohomology of $X_\frak{c}$ admits a Hecke equivariant direct sum decomposition
	\begin{equation}\label{decomposition}
	\mrm{H}^i(X_\frak{c},\bb{C})\cong 
		 \mrm{H}^i_\mrm{uni}(X_\frak{c},\bb{C})\oplus \mrm{H}^i_\mrm{inf}(X_\frak{c},\bb{C})\oplus \mrm{H}^i_\mrm{cusp}(X_\frak{c},\bb{C}) \qquad \forall\ i.
	\end{equation}
Consider $A_{/F}$ a modular elliptic curve of conductor $\frak{f}_A=p_S\cdot\frak{n}^{\mbox{\tiny $+$}}\cdot\frak{n}^{\mbox{\tiny $-$}}$ with associated automorphic representation $\pi$ of $B^\times$. For any subset $\Sigma\subseteq S$ we write $p_\Sigma=\prod_{\frak{p}\in\Sigma}\frak{p}$ and, for notational convenience, we denote by $\Gamma_\Sigma$ the arithmetic subgroup of $B^\times$ of level $p_\Sigma\cdot\frak{n}^{\mbox{\tiny $+$}}$.

\begin{corollary}\label{dimone}
	The $\pi$-isotypic component of the group cohomology of $\Gamma_\Sigma$ satisfies
	\[
	\mrm{dim}_\bb{Q}\ \mrm{H}^{\bfcdot}(\Gamma_\Sigma,\bb{Q})_{\pi}=0\quad \text{if}\ \ \Sigma\not=S,\qquad \mrm{dim}_\bb{Q}\ \mrm{H}^{n+s}(\Gamma_S,\bb{Q})_{\pi}=1\quad \text{if} \ \ \Sigma=S.
	\]
\end{corollary}
\begin{proof}
For any $\cal{O}_F$-ideal $\frak{c}$, the Hecke equivariant decomposition $\eqref{decomposition}$, and the Weil bounds for the Hecke eigenvalues of $A_{/F}$ imply that that the $\pi$-isotypic component $\mrm{H}^{\bfcdot}(X_\frak{c},\bb{C})_{\pi}$ is all contained in the cuspidal subspace. Then, the claims follow from a combination of the Jacquet--Langlands correspondence, the generalized Eichler--Shimura isomorphism (\cite{ArbitraryDarmon}, Section 2.5) and multiplicity-one for $\mrm{GL}_{2,F}$.
\end{proof}

\section{Bruhat-Tits buildings and harmonic cochains}
For every prime $\frak{p}\in S$ we consider the Bruhat-Tits tree $\scr{T}_\frak{p}$ of $\mrm{PGL}_2(F_\frak{p})$. It is a homogeneous tree whose set of vertices, $\scr{V}_\frak{p}$, consists of homothety classes of $\cal{O}_{F,\frak{p}}$-lattices of $F_\frak{p}\oplus F_\frak{p}$. Two vertices $v_1, v_2\in \scr{V}_\frak{p}$ are connected by an oriented edge $e\in\scr{E}_\frak{p}$ with source $s(e)=v_1$ and target $t(e)=v_2$ if there are representatives $\Lambda_1$, $\Lambda_2$ such that 
\[
\frak{p}\Lambda_2\subsetneq\Lambda_1\subsetneq\Lambda_2.
\]
If $e\in\scr{E}_\frak{p}$, then we denote by $\overline{e}$ the opposite edge, i.e., $s(\overline{e})=t(e)$ and $t(\overline{e})=s(e)$.
The tree $\scr{T}_\frak{p}$ is endowed with a natural left action of $\mrm{PGL}_2(F_\frak{p})$, therefore if we denote by $\mbf{v}_{\frak{p}}$ the vertex associated to the lattice $\cal{O}_{F,\frak{p}}\oplus\cal{O}_{F,\frak{p}}$ and by $\mbf{e}_{\frak{p}}$ the edge going from $\mbf{v}_{\frak{p}}$ to the vertex $\widehat{\mbf{v}}_{\frak{p}}$ associated with $\cal{O}_{F,\frak{p}}\oplus\frak{p}\cal{O}_{F,\frak{p}}$ we can make the identifications 
\[
\scr{V}_\frak{p}=\mrm{PGL}_2(F_\frak{p})/\mrm{PGL}_2(\cal{O}_{F,\frak{p}}),\qquad \scr{E}_\frak{p}=\mrm{PGL}_2(F_\frak{p})/\mrm{Iw}(\frak{p}).
\]
We let $\mrm{PGL}_2(F_\frak{p})^0$ be the group defined by the short exact sequence
\[\xymatrix{
1\ar[r]& \mrm{PGL}_2(F_\frak{p})^0\ar[r]&\mrm{PGL}_2(F_\frak{p})\ar[rr]^-{\mrm{ord}_\frak{p}\circ\det}&&\bb{Z}/2\bb{Z}\ar[r]&1,
}\]
and define the sets of even vertices and even edges by
\[
\scr{V}_\frak{p}^0:=\mrm{PGL}_2(F_\frak{p})^0/\mrm{GL}_2(\cal{O}_{F,\frak{p}}),\qquad\scr{E}_\frak{p}^0:=\mrm{PGL}_2(F_\frak{p})^0/\mrm{Iw}(\frak{p}).
\]

\begin{definition}
The Bruhat-Tits building for $\mrm{PGL}_2(F_S)$ is the product of trees $\scr{T}_S:=\prod_{\frak{p}\in S}\scr{T}_\frak{p}$. For any subset $\Sigma\subseteq S$ we define the set of (even) multivertices, (even) oriented multiedges by 
\[
\scr{V}_\Sigma^\star:=\prod_{\frak{p}\in \Sigma}\scr{V}_\frak{p}^\star,\qquad  \scr{E}^\star_{ \Sigma}:=\prod_{\frak{p}\in \Sigma}\scr{E}^\star_\frak{p}\qquad\text{for}\quad \star\in\{\emptyset,0\}.
\]
The inversion $\overline{\ }^\frak{p}:\scr{E}_\Sigma\rightarrow \scr{E}_\Sigma$ at $\frak{p}\in \Sigma$ is described by $(\overline{e}^\frak{p})_\frak{q}=e_\frak{q}$ if $\frak{q}\not=\frak{p}$ and $(\overline{e}^\frak{p})_\frak{q}=\overline{e_\frak{p}}$ if $\frak{q}=\frak{p}$.
\end{definition}

\noindent We fix once and for all an isomorphism $\iota\colon B_S\overset{\sim}{\rightarrow} \mrm{M}_2(F_S)$ such that
$\iota\colon R_{\cal{O}_F}\otimes_{\cal{O}_F}\cal{O}_{F,S}\overset{\sim}{\rightarrow} \mrm{M}_2(\cal{O}_{F,S})$ and  $\iota\colon R_{p_S}\otimes_{\cal{O}_F}\cal{O}_{F,S}\overset{\sim}{\rightarrow} \mrm{M}_2(p_S\cal{O}_{F,S})$.

\begin{proposition}
	For any subset $\Sigma\subseteq S$, the group $\Gamma$ acts transitively on  $\scr{V}^0_{S\setminus\Sigma}\times\scr{E}^0_{\Sigma}$.
\end{proposition}
\begin{proof}
	Let $R_\Sigma$ denote the Eichler order of level $ p_\Sigma\cdot\frak{n}^{\mbox{\tiny $+$}}$ and $R=R_\Sigma[\mbox{\small $S^{-1}$}]$. By strong approximation, the inclusion $B_S^\times\hookrightarrow \widehat{B}^\times$, of the units of the $S$-adic completion of $B$ into the units of the profinite completion, induces a bijection 
	\[\xymatrix{
	R^\times\backslash B_S^\times/(R_{\Sigma})^\times_S\ar[r]^-{1:1}& B^\times\backslash \widehat{B}^\times/\widehat{R}_\Sigma^\times.
	}\]
	As $B$ is not totally definite and $R_\Sigma$ is locally norm maximal, being an Eichler order, the reduced norm induces a bijection between $B^\times\backslash \widehat{B}^\times/\widehat{R}_\Sigma^\times$ and a quotient of the narrow class group of $F$. Therefore, under our narrow class number one assumption, the set 
	\[\xymatrix{
	R^\times\backslash B_S^\times/(R_\Sigma)_S^\times\ar[r]^-{\iota}_-\sim& R^\times\backslash \mrm{GL}_2(F_S)/\mrm{GL}_2(\cal{O}_{F,S\setminus\Sigma})\mrm{Iw}(p_{\Sigma})
	}\] consists of a single element. We deduce that also 
	\[
	(R^\times_\mrm{ev}/\cal{O}_{F}[\mbox{\small $S^{-1}$}]^\times)\backslash \mrm{PGL}_2(F_S)^0/\mrm{GL}_2(\cal{O}_{F,S\setminus\Sigma})\mrm{Iw}(p_{\Sigma})
	\]
	consists of a single element. To conclude we just need to show that the natural map
	\[
	\Gamma\backslash \mrm{PGL}_2(F_S)^0/\mrm{GL}_2(\cal{O}_{F,S\setminus\Sigma})\mrm{Iw}(p_{\Sigma})\longrightarrow (R^\times_\mrm{ev}/\cal{O}_{F}[\mbox{\small $S^{-1}$}]^\times)\backslash \mrm{PGL}_2(F_S)^0/\mrm{GL}_2(\cal{O}_{F,S\setminus\Sigma})\mrm{Iw}(p_{\Sigma})
	\] 
	is injective. This follows from the fact that the target consists of a single element and that, using (\ref{ses1}) and (\ref{ses2}), any element $g\in R^\times_\mrm{ev}$ can be written as $g=\gamma\cdot \alpha$ for some $\gamma\in \Gamma$ and $\alpha\in R_\Sigma^\times$. 
\end{proof}

\noindent For any set $X$ and abelian group $A$ we write $\cal{F}(X,A)$ for the set of functions from $X$ to $A$. If $X$ is a left $\mrm{PGL}_2(F_S)$-set, we define a left action of $\mrm{PGL}_2(F_S)$ on $\cal{F}(X,A)$ by the rule
\[
\gamma\star c(x):=c(\gamma^{-1}x).
\]

\begin{corollary}\label{ident}
	For any subset $\Sigma\subseteq S$ there are isomorphisms of $\Gamma$-modules 
	\[
	\scr{V}^0_{S\setminus\Sigma}\times\scr{E}^0_{\Sigma}\ \cong\ \Gamma/\Gamma_\Sigma
	\qquad\&\qquad
	\cal{F}\big(\scr{V}^0_{S\setminus\Sigma}\times\scr{E}^0_{\Sigma},\ A\big)\ \cong\ \mrm{CoInd}^\Gamma_{\Gamma_\Sigma}A.
	\]
\end{corollary}

\subsection{Harmonic cochain-valued cohomology classes}
For $\frak{p}\in S$ and $m\in \bb{N}\cup\{\infty\}$ we denote by $\scr{V}_\frak{p}^{\le m}$ the set of vertices of the Bruhat--Tits tree $\scr{T}_\frak{p}$ at distance less than or equal to $m$ from the base vertex $\mbf{v}_{\frak{p}}$. Moreover, we let $\scr{E}_\frak{p}^{\le m}$ denote the set of oriented edges whose farthest vertex has distance at most $m$ from the base vertex $\mbf{v}_{\frak{p}}$. 
\begin{definition}
For any abelian group $A$ we consider
\[
\cal{F}_0\big(\scr{E}_\frak{p}^{\le m},A\big):=\left\{ c\in\cal{F}\big(\scr{E}_\frak{p}^{\le m},A\big)\mid c(e)+c(\overline{e})=0\quad\forall\ e\in\scr{E}_\frak{p}^{\le m}\right\}.
\]
Observe that if $A[\scr{E}_\frak{p}^{0,\le m}]$ denotes the free $A$-module on the set of even vertices $\scr{E}_\frak{p}^{0,\le m}$, then
    \[
    \cal{F}_0\big(\scr{E}_\frak{p}^{\le m},A\big)\cong \mrm{Hom}_{A\mbox{-}\mrm{mod}}\big(A[\scr{E}_\frak{p}^{0,\le m}],A\big).
    \]
  \end{definition}

\noindent We  define the complex 
\begin{equation}\xymatrix{
C^{\bfcdot}_{\frak{p},m}(A):\quad 0\ar[r]& \cal{F}_0\big(\scr{E}_\frak{p}^{\le {m+1}},A\big)\ar[r]^-{\varphi_\frak{p}}& \cal{F}\big(\scr{V}_\frak{p}^{\le m},A\big)\ar[r]&0
}\end{equation}
 where the degeneracy map is given by
\begin{equation}
\varphi_\frak{p}:\cal{F}_0(\scr{E}_\frak{p}^{\le m+1},A)\longrightarrow \cal{F}(\scr{V}_\frak{p}^{\le m}, A),\qquad \varphi_\frak{p}(c)(v)=\sum_{s(e')=v}c(e').
\end{equation}
\begin{definition}\label{def: F 0}
For an abelian group $A$, an index $m\in\bb{N}\cup\{\infty\}$ and a subset $\Sigma\subseteq S$, we define 
\[
\cal{F}_0\big(\scr{V}_{S\setminus\Sigma}^{\le m}\times\scr{E}^{\le m+1}_{\Sigma},A\big)\subseteq \cal{F}\big(\scr{V}^{\le m}_{S\setminus\Sigma}\times\scr{E}^{\le m+1}_{\Sigma},A\big)
\]
 to be the subset consisting of those functions $c$ satisfying
\[
c(v,e)+c(v,\overline{e}^\frak{p})=0\qquad \forall\ \frak{p}\in \Sigma\quad\forall\ (v,e)\in \scr{V}_{S\setminus \Sigma}\times\scr{E}_\Sigma.
\]	
\end{definition}

\noindent  For any $m\in \bb{N}$ we denote by $C^{\bfcdot}_{S,m}(\bb{Q})$ the tensor product $\bigotimes_{\frak{p}\in S}C^{\bfcdot}_{\frak{p},m}(\bb{Q})$ of chain complexes of $\bb{Q}$-vector spaces. As the sets $\scr{V}_\frak{p}^{\le m}$, $\scr{E}_\frak{p}^{\le m+1}$ are finite and $\bb{Q}$ is a field we can fix isomorphisms
	\[
	C^{i}_{S,m}(\bb{Q})\cong \bigoplus_{\Sigma\subseteq S, \lvert \Sigma\rvert=i}\cal{F}_0\big(\scr{V}_{S\setminus\Sigma}^{\le m}\times\scr{E}^{\le m+1}_{\Sigma},\bb{Q}\big) \qquad \forall\ i=0,\dots, \lvert S\rvert.
	\]

\begin{definition}
	We consider the projective limit of chain complexes
	$C^{\bfcdot}_{S,\infty}(\bb{Q}):=\underset{\leftarrow, m}{\lim}\ C^{\bfcdot}_{S,m}(\bb{Q})$
	where the transition maps are given by the natural restriction of functions.
	We define $C^{\bfcdot}_{S,\infty}(\bb{Z})$ as the subcomplex of $C^{\bfcdot}_{S,\infty}(\bb{Q})$ satisfying
	\[
	C^{i}_{S,\infty}(\bb{Z})=\bigoplus_{\Sigma\subseteq S, \lvert \Sigma\rvert=i}\cal{F}_0\big(\scr{V}_{S\setminus\Sigma}\times\scr{E}_{\Sigma},\bb{Z}\big) \qquad \forall\ i=0,\dots, \lvert S\rvert.
	\]
\end{definition}

\begin{proposition}
	The cohomology of the complex $C^{\bfcdot}_{S,\infty}(\bb{Q})$
is concentrated in the top degree, i.e., $\mrm{H}^i\big(C^{\bfcdot}_{S,\infty}(\bb{Q})\big)=0$ when $i<\lvert S\rvert$.
\end{proposition}
\begin{proof}
	First, note that for every $m\in \bb{N}$ the cohomology of the complex $C^{\bfcdot}_{S,m}(\bb{Q})$ is concentrated in the top degree because of Kunneth's formula and the fact that  $\mrm{H}^0\big(C^{\bfcdot}_{\frak{p},m}(\bb{Q})\big)=0$ for every $\frak{p}\in S$ and every $m\in \bb{N}$ (\cite{Greenberg}, Lemma 24). As projective limits are exact on exact sequences of finite dimensional vector spaces, the claim follows.
\end{proof}

\begin{definition}
	Let $A$ be either $\bb{Z}$ or $\bb{Q}$. The space of $A$-valued $S$-harmonic cochains on the Bruhat-Tits building $\scr{T}_S$ is the subspace of $\cal{F}_0(\scr{E}_S,A)$ given by 
	\[
	\mrm{HC}_S(A):=\mrm{H}^{\lvert S\rvert}\big(C^{\bfcdot}_{S,\infty}(A)\big).
	\]
\end{definition}

\begin{corollary}\label{corollincliso}
	The inclusion $\mrm{HC}_S(\bb{Q})\hookrightarrow\cal{F}_0(\scr{E}_S,\bb{Q})$ induces an isomorphism
	\[
	 \mrm{H}^{n+s}\big(\Gamma,\mrm{HC}_S(\bb{Q})\big)_\pi\overset{\sim}{\longrightarrow} \mrm{H}^{n+s}\big(\Gamma_S,\bb{Q}\big)_\pi
	\]
	of one dimensional eigenspaces.
\end{corollary}
\begin{proof}
	We extract  the short exact sequence 
		\[\xymatrix{
		0\ar[r]& \mrm{HC}_S(\bb{Q})\ar[r]&\cal{F}_0(\scr{E}_S,\bb{Q})\ar[r]& \mrm{coker}(\bb{Q})\ar[r]&0
		}\]
		from the complex $C^{\bfcdot}_{S,\infty}(\bb{Q})$ of $\Gamma$-modules. By Corollary $\ref{ident}$ and Shapiro's lemma 
\[\begin{split}
\mrm{H}^{\bfcdot}\big(\Gamma, \cal{F}_0(\scr{V}_{S\setminus\Sigma}\times\scr{E}_{\Sigma}, \bb{Q})\big)_\pi
&\cong
\mrm{H}^{\bfcdot}\big(\Gamma, \cal{F}(\scr{V}_{S\setminus\Sigma}\times\scr{E}^0_{\Sigma}, \bb{Q})\big)_\pi\\ &\cong\big[\mrm{H}^{\bfcdot}\big(\Gamma_\Sigma,\bb{Q}\big)_\pi\big]^{\oplus 2^{\lvert S\setminus \Sigma\rvert}}.
\end{split}\]
In particular,  the cohomology groups are trivial whenever $\Sigma\not=S$ by Corollary $\ref{dimone}$. Therefore, one sees that $\mrm{H}^{\bfcdot}(\Gamma,\mrm{coker}(\bb{Q}))_\pi=0$ using an inductive argument on the terms of the complex $C^{\bfcdot}_{S,\infty}(\bb{Q})$. The claim follows because
\[
\mrm{H}^{n+s}\big(\Gamma, \cal{F}_0(\scr{E}_{S}, \bb{Q})\big)_\pi\cong \mrm{H}^{n+s}\big(\Gamma_S,\bb{Q}\big)_\pi
\]
has dimension one by Corollary $\ref{dimone}$.
\end{proof}

\begin{lemma}\label{cohclass}
	We have $\mrm{H}^{\bfcdot}\big(\Gamma,\mrm{HC}_S(\bb{Q})\big)\cong \mrm{H}^{\bfcdot}\big(\Gamma,\mrm{HC}_S(\bb{Z})\big)\otimes\bb{Q}$.
	Hence, there is a non-torsion class  
	\[
	\mbf{c}_A\in \mrm{H}^{n+s}(\Gamma,\mrm{HC}_S(\bb{Z}))_\pi
	\]
well-defined up to torsion and up to sign.
\end{lemma}
\begin{proof}
	By Shapiro's lemma and the universal coefficient theorem, for every $\Sigma\subseteq S$ we have
\[\xymatrix{
\mrm{H}^{\bfcdot}\big(\Gamma, \cal{F}_0(\scr{V}_{S\setminus\Sigma}\times\scr{E}_{\Sigma}, \bb{Z})\big)\otimes\bb{Q}\ar[r]^-{\sim}\ar[d]& \mrm{H}^{\bfcdot}\big(\Gamma_\Sigma,\bb{Z}\big)^{\oplus 2^{\lvert S\setminus\Sigma\rvert}}\otimes\bb{Q}\ar[d]^-{\sim}\\
\mrm{H}^{\bfcdot}\Big(\Gamma, \cal{F}_0(\scr{V}_{S\setminus\Sigma}\times\scr{E}_{\Sigma}, \bb{Q})\big)\ar[r]^-{\sim}& \mrm{H}^{\bfcdot}\big(\Gamma_\Sigma,\bb{Q}\big)^{\oplus 2^{\lvert S\setminus\Sigma\rvert}}.
}\]
Therefore, $\mrm{H}^{\bfcdot}\big(\Gamma, \cal{F}_0(\scr{V}_{S\setminus\Sigma}\times\scr{E}_{\Sigma}, \bb{Z})\big)\otimes\bb{Q}\cong \mrm{H}^{\bfcdot}\big(\Gamma, \cal{F}_0(\scr{V}_{S\setminus\Sigma}\times\scr{E}_{\Sigma}, \bb{Q})\big)$.
	As in the proof of Corollary \ref{corollincliso}, we can extract from the complex $C^{\bfcdot}_{S,\infty}(\bb{Z})$ a short exact sequence 
		\[\xymatrix{
		0\ar[r]& \mrm{HC}_S(\bb{Z})\ar[r]&\cal{F}_0(\scr{E}_S,\bb{Z})\ar[r]& \mrm{coker}(\bb{Z})\ar[r]&0.
		}\]
		By an inductive argument on the terms of the complex $C^{\bfcdot}_{S,\infty}(\bb{Z})$, we have $\mrm{H}^{\bfcdot}(\Gamma,\mrm{coker}(\bb{Z}))\otimes\bb{Q}\cong\mrm{H}^{\bfcdot}(\Gamma,\mrm{coker}(\bb{Q}))$. The claim follows.
\end{proof}

\section{The integration pairing}
 Consider the natural  $\mrm{GL}_2(F_S)$-equivariant projection
\[
\mrm{pr}: \prod_{\frak{p}\in S}\big(F_\frak{p}^2\setminus\{\underline{0}\}\big)\longrightarrow \bb{P}^1(F_S),\qquad (x,y)\mapsto [x:y].
\]
For a multiedge $e=(s,t)\in\scr{E}_S$, choose lattices $\Lambda_s=(\Lambda_{s_\frak{p}})_\frak{p}$ and $\Lambda_t=(\Lambda_{t_\frak{p}})_\frak{p}$ such that 
\[
\frak{p}\Lambda_{t_\frak{p}}\subsetneq\Lambda_{s_\frak{p}}\subsetneq\Lambda_{t_\frak{p}}\qquad  \forall\ \frak{p}\mid p. 
\] For any lattice $\Lambda\subset F_S\oplus F_S$ we let $\Lambda'=(\Lambda'_\frak{p})_\frak{p}$ be given by $\Lambda'_\frak{p}=\Lambda_\frak{p}\setminus\frak{p}\Lambda_\frak{p}$ for all $\frak{p}\mid p$. Then 
\[
U_e:=\mrm{pr}(\Lambda_s'\cap \Lambda_t')
\] 
is an open compact subset of $\bb{P}^1(F_S)$ depending only on the multiedge $e\in\scr{E}_S$. Moreover, the collection $\{U_e\}_{e}$, indexed by  $\scr{E}_S$, forms a basis of the $p$-adic topology of $\bb{P}^1(F_S)$. 
\begin{remark}
	For any $\gamma\in\mrm{PGL}_2(F_S)$ and any multiedge $e\in\scr{E}_S$ we have
	\[
	U_{\gamma\cdot e}=\mrm{pr}\left(\gamma\Lambda_s'\cap\gamma\Lambda_t'\right)=\gamma\cdot\mrm{pr}\left(\Lambda_s'\cap\Lambda_t'\right)=\gamma\cdot U_e.
	\]
\end{remark}

\begin{definition}
We denote by $\scr{M}_\plectic(\bb{P}^1(F_S),\bb{Z})$ the set of finitely additive $\bb{Z}$-valued functions $\mu$ on compact open subsets of $\bb{P}^1(F_S)$  satisfying the following property:
the value $\mu(U)$ equals zero if there exists $\frak{p}\in S$ such that
	\begin{equation}\label{measurezero sets}
	U=\bb{P}^1(F_\frak{p})\times V,\qquad\text{for}\qquad V\subseteq\bb{P}^1(F_{S\setminus\{\frak{p}\}})\ \text{compact open}.
	\end{equation}
\end{definition}

\begin{lemma}\label{cocyclestodistributions}
	There is a $\mrm{PGL}_2(F_S)$-equivariant isomorphism 
	\[
	\mrm{HC}_S(\bb{Z})\overset{\sim}{\longrightarrow} \scr{M}_\plectic(\bb{P}^1(F_S),\bb{Z}),\qquad c\mapsto\mu_c
	\] 
	where the measure $\mu_c$ is characterized by $\mu_c(U_e)=c(e)$ for any multiedge $e\in\scr{E}_S$.
\end{lemma}
\begin{proof}
	This is standard when $\lvert S\rvert=1$, the general case is a straightforward modification.
\end{proof}

\subsection{Multiplicative integrals}
\begin{definition}
 We denote by $\widehat{E}_\frak{p}^\times$ the free part of the $p$-adic completion of $E_\frak{p}^\times$ and consider
	\[
	\widehat{E}_{S,\otimes}^\times :=\bigotimes_{\frak{p}\in S} \widehat{E}_\frak{p}^\times
	\]
 the tensor product of finite free $\bb{Z}_p$-modules.
 \end{definition}
\noindent The multiplicative integral of a continuous function $f:\bb{P}^1(F_S)\rightarrow \widehat{E}_{S,\otimes}^\times$ with respect to a measure  $\mu\in\scr{M}_\plectic(\bb{P}^1(F_S),\bb{Z})$ is defined as
		\begin{equation}
		\mint_{\bb{P}^1(F_S)}f(t)\ \mrm{d}\mu(t)
		:=
		\underset{\scr{U}}{\lim}\ \prod_{U\in\scr{U}} f(t_U)^{\mu(U)}
		\end{equation}
		where the limit is taken over increasingly finer coverings $\scr{U}$ of $\bb{P}^1(F_S)$ by disjoint compact open subsets, and $t_U\in U$. The definition is well-posed because the target $\widehat{E}_{S,\otimes}^\times$ is $p$-adically complete.

\begin{lemma}\label{trivial integrals}
	Suppose there are continuous functions $f_\frak{p}:\bb{P}^1(F_\frak{p})\to\widehat{E}_\frak{p}^\times$ such that $f=\otimes_\frak{p}f_\frak{p}$. If at least one of the functions is constant, then 
	\[
	\mint_{\bb{P}^1(F_S)}f(t)\ \mrm{d}\mu(t)=1.
	\]
\end{lemma}
\begin{proof}
	Suppose that the $\frak{p}$-th component $f_{\frak{p}}$ is constant. Any covering of $\bb{P}^1(F_S)$ by disjoint compact open subsets can be refined to one of the form $\scr{U}=\scr{U}_{\frak{p}}\times\scr{U}^{\frak{p}}$, where $\scr{U}_{\frak{p}}$ is a covering by disjoint compact opens of $\bb{P}^1(F_{\frak{p}})$ and $\scr{U}^{\frak{p}}$ is a covering by disjoint compact opens of $\bb{P}^1(F_{S\setminus\{\frak{p}\}})$. Then ($\ref{measurezero sets}$) gives the result because
	\[
	\prod_{U\in\scr{U}^{\frak{p}}}\prod_{V\in\scr{U}_{\frak{p}}}f(t_U,t_V)^{\mu(U\times V)}=\prod_{U\in\scr{U}^{\frak{p}}}f(t_U,t_V)^{\mu(U\times \bb{P}^1(F_{\frak{p}}))}=1.
	\]
\end{proof}

\noindent We obtain $\mrm{PGL}_2(F_S)$-equivariant integration pairing
\[
\mint\colon \scr{M}_\plectic(\bb{P}^1(F_S),\bb{Z})\times \bigotimes_{\frak{p}\in S}\Big(\scr{C}(\bb{P}^1(F_\frak{p}),\widehat{E}_\frak{p}^\times)^\times/\widehat{E}_\frak{p}^\times\Big)\longrightarrow \widehat{E}_{S,\otimes}^\times,
\]
which points to the following definition of zero-cycles: we consider the $S$-adic upper half plane $\cal{H}_S=\prod_{\frak{p}\in S}\cal{H}_\frak{p}$ where $\cal{H}_\frak{p}=\bb{P}^1(E_\frak{p})\setminus \bb{P}^1(F_\frak{p})$. We define the group of \emph{plectic zero-cycles} of $\cal{H}_S$ as the subgroup of $\bb{Z}[\cal{H}_S]$ given by
	\begin{equation}
	\bb{Z}_\plectic[\cal{H}_S]:=\bigotimes_{\frak{p}\in S}\bb{Z}[\cal{H}_\frak{p}]^0,
	\end{equation}
	the tensor product of zero-cycles of degree zero for each component. The natural inclusions $\bb{Z}[\cal{H}_\frak{p}]^0\hookrightarrow \scr{C}(\bb{P}^1(F_\frak{p}),\widehat{E}_\frak{p}^\times)^\times/\widehat{E}_\frak{p}^\times$ can be combined to give a $\mrm{PGL}_2(F_S)$-equivariant homomorphism
\[\xymatrix{
\bb{Z}_\plectic[\cal{H}_S]\ar@{^{(}->}[r]&\bigotimes_{\frak{p}\in S}\Big(\scr{C}(\bb{P}^1(F_\frak{p}),\widehat{E}_\frak{p}^\times)^\times/\widehat{E}_\frak{p}^\times\Big).
}\]
Thus, the integration map
$\mint\colon \scr{M}_\plectic(\bb{P}^1(F_S),\bb{Z})\times \bb{Z}_\plectic[\cal{H}_S]\rightarrow \widehat{E}_{S,\otimes}^\times$
produces a cap product pairing 
\begin{equation}
\cap\colon \mrm{H}^{n+s}\big(\Gamma, \scr{M}_\plectic(\bb{P}^1(F_S),\bb{Z})\big)\times \mrm{H}_{n+s}\big(\Gamma,\bb{Z}_\plectic[\cal{H}_S]\big) \longrightarrow \widehat{E}_{S,\otimes}^\times.
\end{equation}

\section{Homology class}\label{Divisor-valued}
Under our running assumptions, there is an an optimal embedding $\psi$ of level $\frak{n}^{\mbox{\tiny $+$}}$, i.e. an $F$-algebra homomorphism $\psi\colon E\to B$ satisfying
$\psi(E)\cap R_{\frak{n}^{\mbox{\tiny $+$}}}=\psi(\cal{O}_E)$.
We denote by $\cal{O}_{1}^\times$ the free part of the subgroup of units $u\in \cal{O}_E^\times$ of norm $\mrm{N}_{E/F}(u)=1$. By Dirichlet's unit theorem
\[
\text{rank}_\bb{Z}\ \cal{O}_{1}^\times=n+s,
\]
and we can choose a generator $\theta$ of the homology group $\mrm{H}_{n+s}(\cal{O}_{1}^\times,\bb{Z})\cong \bb{Z}$. The units $R_{\frak{n}^{\mbox{\tiny $+$}}}^\times$ of the Eichler order act on the embedding $\psi$ by conjugation, they induce an action of $\Gamma_{\frak{n}^{\mbox{\tiny $+$}}}$ such that
\begin{equation}\label{finalidentification}
\text{Stab}_{\Gamma_{\frak{n}^{\mbox{\tiny $+$}}}}(\psi)\cong \cal{O}_{1}^\times.
\end{equation}
We denote by $\theta_\psi$ the generator of $\mrm{H}_{n+s}(\text{Stab}_{\Gamma_{\frak{n}^{\mbox{\tiny $+$}}}}(\psi),\bb{Z})$ corresponding to $\theta$ under ($\ref{finalidentification}$).
 Recall that the quaternion algebra $B/F$ is split at all the primes in $S$, therefore it acts on the $S$-adic upper half plane $\cal{H}_S$ through the fixed isomorphism $\iota:B_S^\times\overset{\sim}{\to}\mrm{GL}_2(F_S)$. The induced action of $\psi(E)^\times$ on $\cal{H}_S$ has two fixed points $\tau_\psi, \bar{\tau}_\psi\in \cal{H}_S$ which we can use to define the plectic zero-cycle 
	\[
	C_\psi:=\otimes_{\frak{p}\in S}([\tau_{\psi,\frak{p}}]-[\bar{\tau}_{\psi,\frak{p}}])\in \bb{Z}_\plectic[\cal{H}_S].
	\]
The isomorphism induced by Shapiro's lemma  $\mrm{H}_{n+s}\big(\text{Stab}_{\Gamma_{\frak{n}^{\mbox{\tiny $+$}}}}(\psi),\bb{Z}\big)\overset{\sim}{\rightarrow}\mrm{H}_{n+s}\big(\Gamma_{\frak{n}^{\mbox{\tiny $+$}}},\bb{Z}[\Gamma_{\frak{n}^{\mbox{\tiny $+$}}}\cdot C_\psi]\big)$,
can be used use to define a map
\[
\iota_\psi:\mrm{H}_{n+s}\big(\text{Stab}_{\Gamma_{\frak{n}^{\mbox{\tiny $+$}}}}(\psi),\bb{Z}\big)\longrightarrow \mrm{H}_{n+s}\big(\Gamma_{\frak{n}^{\mbox{\tiny $+$}}},\bb{Z}_\plectic[\cal{H}_S]\big).
\]
Recall that for simplicity we assumed the number field $E$ to have narrow class number one, then we are entitled to define the homology class associated to the quadratic extension  $E/F$ by
\begin{equation}
\Delta_E:=\mrm{res}\circ\iota_\psi(\theta_\psi)\ \in\ \mrm{H}_{n+s}\big(\Gamma,\bb{Z}_\plectic[\cal{H}_S]\big).
\end{equation}

\begin{definition}
	The \emph{plectic $p$-adic invariant} associated to the triple $(A_{/F},E,S)$ is given by
		\[
		\mrm{Q}_{A}:=\mbf{c}_A\cap\Delta_E\ \in\ \widehat{E}_{S,\otimes}^\times.
		\]
\end{definition}
\noindent We refer to Section \ref{conjectures} of the Introduction for a discussion of the conjectural properties of plectic invariants and their significance for the arithmetic of higher rank elliptic curves.

%
%
%

 \section{The algorithm}\label{algorithm}
For simplicity, we describe the algorithm used to compute numerical approximations of plectic $p$-adic invariants in the setting relevant for the computations. We consider $F$ a real quadratic field, $p=\frak{p}_1\cdot\frak{p}_2$ a split prime, and $A_{/F}$ a modular elliptic curve of conductor $\frak{f}_A= \frak{p}_1\cdot \frak{p}_2\cdot\frak{q}$ for some other prime ideal $\frak{q}$.  We consider an ATR extension $E/F$ where all primes dividing $\frak{f}_A$ are inert, and where one archimedean place $\infty_1$ of $F$ splits, while the other $\infty_2$ does not. Following the recipy, we let $B/F$ denote the quaternion algebra over $F$ ramified at $\{\frak{q},\infty_2\}$.  We continue to use the notations $\Gamma_S, \Gamma_{\frak{p}_1}, \Gamma_{\frak{p}_2}$ for the arithmetic subgroups of $B^\times$ of respective levels $p, \frak{p}_1, \frak{p}_2$ and $\Gamma$ for the $S$-arithmetic subgroup.

\noindent The computation of the homology class $\Delta_E$ is straightforward using the Magma routines that compute optimal embeddings $\psi$ of $\cal{O}_E$ into a maximal order of $B$: 
	  we compute a non-torsion element $u\in\mathcal{O}_E^\times$  of relative norm one, and find a fixed point $\tau\in \mathcal{H}_p$ for the action of $\psi(u)$. Then
      \[
        \Delta_E = ([\tau_{\frak{p}_1}]-[\bar\tau_{\frak{p}_1}])\otimes ([\tau_{\frak{p}_2}]-[\bar\tau_{\frak{p}_2}])\otimes\psi(u).
      \]
	  The main difficulty in computing plectic invariants consists in explicitly describing the class 
	  \[
	  \mbf{c}_A\in \mrm{H}^1(\Gamma,\mrm{HC}_S(\bb{Z}))_\pi
	  \]
	  in terms of the cohomology class $\kappa_A\in \mrm{H}^1(\Gamma_S,\bb{Z})_\pi$ associated to the elliptic curve $A_{/F}$ by modularity. Assuming momentarily the we can compute an approximation $\mbf{c}^m_A$ of $\mbf{c}_A$ for an integer $m\ge0$ we explain how to compute some logarithm of the plectic invariant.  We let $q_{\frak{p}}\in F_{\frak{p}}^\times$ denote the period of the Tate curve $A_{/F_{\frak{p}}}$ for $\frak{p}\in S$, and let $E_{S,\otimes}$ denote the tensor product $\otimes_{\frak{p}\in S}E_{\frak{p}}$ of $\bb{Q}_p$-vector spaces; then there is a homomorphism $ \log_{S}\colon\widehat{E}^\times_{S,\otimes}\rightarrow E_{S,\otimes}$ -- factoring through Tate's uniformization $\log_{S}=\log_A\circ\ \phi_{\mbox{\tiny $\mrm{Tate}$}}$ -- such that 
	 \[	  \log_{S}(\mrm{Q}_{A})=\int_{\bb{P}^1(F_S)}\bigotimes_{\frak{p}\in S} \log_{q_{\frak{p}}} \left(\frac{t_{\frak{p}}-\tau_{\frak{p}}}{t_{\frak{p}}-\bar\tau_{\frak{p}}} \right)\ \mrm{d}\mu_A(t).
	 \]
	 As the induced morphism $\log_A:\widehat{A}(E_{S})\to E_{S,\otimes}$ is injective, we do not loose information about $\phi_{\mbox{\tiny $\mrm{Tate}$}}(\mrm{Q}_{A})$ in computing the additive $p$-adic integral for $\log_{S}(\mrm{Q}_{A})$. 
Concretely, we numerically calculate the Riemann sum
      \[
        \sum_{U_e\in \scr{U}_m} \bigotimes_{\frak{p}\in S} \log_{q_{\frak{p}}} \left(\frac{t_{e,\frak{p}}-\tau_{\frak{p}}}{t_{e,\frak{p}}-\bar\tau_{\frak{p}}} \right)\cdot \mbf{c}_A^m(e),
      \]
      where $\scr{U}_m$ denotes the covering of $\mathbb{P}^1(F_S)$ given by the compact opens $\{U_e\}_{e\in \scr{E}_S^{\le m}}$, and $t_e\in U_e$.

\noindent In the rest of the section we explain how to make Corollary \ref{corollincliso} explicit.

 \paragraph{\textbf{Step 1}.} We begin by computing the cohomology class $\kappa_A\in \mrm{H}^1(\Gamma_S,\bb{Z})_\pi$. We use Magma to calculate a presentation of $\Gamma_S$ in terms of generators and relations, then the algorithmic solution to the word problem provided by a Magma routine allows us to compute some Hecke operators and diagonalize $\mrm{H}^1(\Gamma_S,\bb{Z})$ with respect to their action. We  do this iteratively until we find the $1$-dimensional subspace onto which the Hecke operators have the correct eigenvalues associated with $A_{/F}$. A generator of this subspace is $\kappa_A$. Furthermore, we can identify this class with a 1-cocycle $\kappa_A\in Z^1(\Gamma_S,\bb{Z})$ because there are no non-trivial 1-coboundaries, then $\kappa_A$ is stored in memory in terms of its values on the set of generators of $\Gamma_S$.
  
 \paragraph{\textbf{Step 2}.} 
  Shapiro's lemma gives an explicit Hecke-equivariant isomorphism
  \[ \mrm{H}^1(\Gamma_S,\bb{Z})\overset{\sim}{\longrightarrow}\mrm{H}^1(\Gamma,\cal{F}_0(\scr{E}_S,\bb{Z}))
  \]
  which associates to $\kappa_A$ a 1-cocycle $c_A\in Z^1(\Gamma, \mathcal{F}_0(\scr{E}_{S},\bb{Z}))$.
Even though the group $\Gamma$ is also finitely presented -- with generators obtained easily from the presentations of arithmetic groups -- it is not possible to store the 1-cocycle $c_A$ in a computer because $\scr{E}_{S}$ is infinite. Nevertheless, it suffices to compute an approximation of $c_A$, i.e. its restriction to $\scr{E}_{S}^{\le m} = \scr{E}_{{\frak p}_1}^{\le m}\times\scr{E}_{{\frak p}_2}^{\le m}$ for some $m\ge0$
  \[
{c}_A^m\in Z^1(\Gamma, \mathcal{F}_0(\scr{E}_{S}^{\le m},\bb{Z})).
\]
This computation relies on the algorithms of \cite[Section 4]{GM-Oc}. From a system of coset representatives for $\Gamma_{\frak{p}_1}\backslash \Gamma_1$, we construct a collection of elements in $\Gamma$ as in \cite[Definition 2.3]{GM-Oc}
\[
\{\gamma_{e_1}\}_{{e_1}\in \scr{E}_{{\frak p}_1}^{0,\le m}}\qquad\&\qquad\{\gamma_{v_1}\}_{{v_1}\in \scr{V}_{{\frak p}_1}^{\le m-1}}.
\]
 We do the same for coset representatives of $\Gamma_S\backslash \Gamma_{\frak{p}_1}$ obtaining  $\{\gamma_{e_2}\}_{{e_2}\in \scr{E}_{{\frak p}_2}^{0,\le m}}$ and $ \{\gamma_{v_2}\}_{{v_2}\in \scr{V}_{{\frak p}_2}^{\le m-1}}$.
We are now ready to compute the 1-cocycle $c_A^m$: 
given an element $g\in\Gamma$ and an even multiedge $e=(e_1,e_2)\in\scr{E}_{S}^{0,\le m}$, we first compute an element $b\in \Gamma_{\frak{p}_1}$ such that
$\gamma_{e_1}\cdot g = b\cdot\gamma_{g^{-1}(e_1)}$ -- using the algorithm of \cite[Theorem 4.1]{GM-Oc}.
 Applying the same algorithm a second time, we compute $h\in\Gamma_S$ such that $\gamma_{e_2}\cdot b = h\cdot \gamma_{b^{-1}(e_2)}$, then the 1-cocycle $c_A^m$ is explicitly given by the formula
   \[
     c_A^m(g)(e) = \kappa_A(h).
     \]

\begin{remark}
 Corollary \ref{corollincliso} ensures that $c_A\in Z^1(\Gamma,\cal{F}_0(\scr{E}_S,\bb{Z}))$ represents a  \emph{cohomology class} valued in harmonic cochains. However, the 1-cocycle $c_A$ might not belong to $Z^1(\Gamma,\mrm{HC}_S(\bb{Z}))$ in general. To better understand the failure of harmonicity, recall the homomorphism
 	\[\xymatrix{
 	\cal{F}_0\big(\scr{E}_{S}, \bb{Z}\big)\ar[rr]^-{\varphi_1\oplus\varphi_2}&& \cal{F}_0\big(\scr{V}_{\frak{p}_1}\times\scr{E}_{\frak{p}_2}, \bb{Z}\big)\oplus \cal{F}_0\big(\scr{E}_{\frak{p}_1}\times\scr{V}_{\frak{p}_2}, \bb{Z}\big)
 	}\]
 	given by 
 	\[
 	\varphi_1(c)(v_1,e_2)=\sum_{s(e_1)=v_1}c(e_1,e_2),\qquad \varphi_2(c)(e_1,v_2)=\sum_{s(e_2)=v_2}c(e_1,e_2).
 	\]
  The collection $\{\gamma_{e_1}\}_{{e_1}}$ is a radial system \cite[Definition 2.3]{GM-Oc}, hence $\varphi_1(c_A^m)=0$.  However, we might have $\varphi_2(c_A^m)\neq 0$ (even if the second collection  $\{\gamma_{e_2}\}_{{e_2}}$ is radial for the tree in the second variable) when elements of the two collections do not commute. Fortunately, as we do know  that $c_A^m$ is cohomologous to a 1-cocycle with values in harmonic cochains, we just need to find some $D\in \mathcal{F}_0(\scr{E}_{S}^{\le m},\bb{Z})$ such that the modification 
        \[
     \mbf{c}_A^m :=  c_A^m - \partial D
        \]
is harmonic. Here $\partial D$ denotes the 1-coboundary arising from $D$. 	
 \end{remark} 
 
  \paragraph{\textbf{Step 3}.}
 The final step consists in the computation of  a function $D\in \mathcal{F}_0(\scr{E}_{S}^{\le m},\bb{Z})$ such that $ \mbf{c}_A^m =  c_A^m - \partial D$ is harmonic. Since $c_A$ represents a cohomology class valued in harmonic cochains, there are two elements
    $ D_1 \in\mathcal{F}_0( \scr{V}_{\frak{p}_1}^{\leq m-1}\times\scr{E}_{\frak{p}_2}^{\leq m},\bb{Z})$ and  $D_2 \in\mathcal{F}_0( \scr{E}_{\frak{p}_1}^{\leq m}\times\scr{V}_{\frak{p}_2}^{\leq m-1},\bb{Z})$
satisfying
 \[
   \varphi_1(c_A^m)(g) = (g-1)\star D_1\qquad\&\qquad 
   \varphi_2(c_A^m)(g) =(g-1)\star D_2.
   \]
  Moreover, the function $D$ has to satisfy $\varphi_i(D)=D_i$ for $i=1,2$. First, we explain how to compute $D_1$ and $D_2$. By Definition \ref{def: F 0} it suffices to calculate their values at even edges. Since $\Gamma$ acts transitively on $\scr{V}^0_{\frak{p}_1}\times\scr{E}^0_{\frak{p}_2}$, any pair $(v_1,e_2)\in \scr{V}_{\frak{p}_1}\times\scr{E}^0_{\frak{p}_2}$ is of the form $g^{-1}(\mbf{v}_{\frak{p}_1},\mbf{e}_{\frak{p}_2})$ or $g^{-1}(\widehat{\mbf{v}}_{\frak{p}_1},\mbf{e}_{\frak{p}_2})$ for some $g\in \Gamma$. Similarly, any pair $(e_1,v_2)\in\scr{E}^0_{\frak{p}_1}\times \scr{V}_{\frak{p}_2}$ is of the form $g^{-1}(\mbf{e}_{\frak{p}_1},\mbf{v}_{\frak{p}_2})$ or $g^{-1}(\mbf{e}_{\frak{p}_1},\widehat{\mbf{v}}_{\frak{p}_2})$ for some $g\in \Gamma$. Therefore, the values of $D_1$ are computed by either
  \[
   D_1(v_1,e_2) =  [\varphi_1(c_A^m)(g) + D_1](\mbf{v}_{\frak{p}_1},\mbf{e}_{\frak{p}_2})\qquad\text{or}\qquad
   D_1(v_1,e_2) =  [\varphi_1(c_A^m)(g) + D_1](\widehat{\mbf{v}}_{\frak{p}_1},\mbf{e}_{\frak{p}_2}),
 \]
and the values of $D_2$ by either
  \[
   D_2(e_1,v_2) =  [\varphi_2(c_A^m)(g) + D_2](\mbf{e}_{\frak{p}_1},\mbf{v}_{\frak{p}_2})\qquad
\text{or}\qquad
   D_2(e_1,v_2) =  [\varphi_2(c_A^m)(g) + D_2](\mbf{e}_{\frak{p}_1},\widehat{\mbf{v}}_{\frak{p}_2}).
 \]
In other words, $D_1$ and $D_2$ are completely determined by the degenerations $\varphi_1(c_A^m)$,  $\varphi_2(c_A^m)$ -- which can be effectively computed -- and by the four quantities
 \[
D_1(\mbf{v}_{\frak{p}_1},\mbf{e}_{\frak{p}_2}),\ D_1(\widehat{\mbf{v}}_{\frak{p}_1},\mbf{e}_{\frak{p}_2}),\ D_2(\mbf{e}_{\frak{p}_1},\mbf{v}_{\frak{p}_2}),\ D_2(\mbf{e}_{\frak{p}_1},\widehat{\mbf{v}}_{\frak{p}_2}).
   \]
These four quantities are then pinned down by the fact that the pair $(D_1,D_2)$ is in the kernel of
	\[\begin{array}{ccc}
 \cal{F}_0\big(\scr{V}_{\frak{p}_1}\times\scr{E}_{\frak{p}_2}, \bb{Z}\big)\oplus \cal{F}_0\big(\scr{E}_{\frak{p}_1}\times\scr{V}_{\frak{p}_2}, \bb{Z}\big)& {\longrightarrow} &\cal{F}\big(\scr{V}_{\frak{p}_1}\times\scr{V}_{\frak{p}_2}, \bb{Z}\big)\\ (c_1,c_2) & \longmapsto & \nu_2(c_1)-\nu_1(c_2)\end{array}
 	\]
 	given by 
 	\[
 	\nu_1(c_2)(v_1,v_2)=\sum_{s(e_1)=v_1}c_2(e_1,v_2),\qquad \nu_2(c_1)(v_1,v_2)=\sum_{s(e_2)=v_2}c_1(v_1,e_2).
 	\]
 More precisely, for every $(v_1,v_2)\in \{  (\mbf{v}_{\frak{p}_1},\mbf{e}_{\frak{p}_2}),\ (\widehat{\mbf{v}}_{\frak{p}_1},\mbf{e}_{\frak{p}_2}),\ (\mbf{e}_{\frak{p}_1},\mbf{v}_{\frak{p}_2}),\ (\mbf{e}_{\frak{p}_1},\widehat{\mbf{v}}_{\frak{p}_2})\}$ the values satisfy
        \[
          [\nu_2(D_1)-\nu_1(D_2)](v_1,v_2) = 0.
        \]
        We are left to explain the computation of a function $D\in \cal{F}_0(\scr{E}_S^{\le m},\bb{Z})$ lifting $(D_1,D_2)$.
We determine $D$ by solving the system of linear equations given by
   \begin{itemize}
   	\item [$\bfcdot$] $(\varphi_1D)(v_1,e_2) = D_1(v_1,e_2)$ for every $(v_1,e_2)\in \scr{V}_{\frak{p}_1}^{0,\leq m-1}\times \scr{E}_{\frak{p}_2}^{0,\leq m}$,
		\item [$\bfcdot$]   $(\varphi_2D)(e_1,v_2) = D_2(e_1,v_2)$ for every
$(e_1,v_2)\in \scr{E}_{\frak{p}_1}^{0,\leq m}\times \scr{V}_{\frak{p}_2}^{0,\leq m-1}$.
   \end{itemize}     
The coefficient matrix of this system quickly becomes very sparse: the number of columns is $O(p^{2m})$ and there are only $p+1$ non-zero coefficients in each of the $O(p^{2m-1})$ row. In order to solve the system efficiently, we crucially use an algorithm exploiting the sparsity of the matrix to both speed up the computation and reduce the memory usage.

\section{Numerical experiments}\label{Lastsection}
Let $F$ be one of the two smallest real quadratic fields of narrow class number one, where the prime $3$ splits -- i.e. $(3)=\frak{p}_1\cdot\frak{p}_2$. We write $D_F$ for the discriminant and set $w=(1+\sqrt{D_F})/2$.  We consider $A_{/F}$  semistable elliptic curves of conductor $\frak{f}_A=\frak{p}_1\cdot\frak{p}_2\cdot\frak{q}$, and denote by $\varepsilon_\frak{q}$ the root number at $\frak{q}$. Moreover, we choose quadratic ATR extensions $E=F(\sqrt{\beta})$ where all primes divisors of $\frak{f}_A$ are inert. In particular, $r_\mrm{alg}(A/E)$ will always be even, and by choosing $S=\{\frak{p}_1,\frak{p}_2\}$ we compute the quantity $\log_S(\mrm{Q}_A)$ associated to the plectic $3$-adic invariant ``for rank two''.
 We test Conjectures \ref{algebraicity}, \ref{Kolyvagin} and \ref{lowerrank} by carrying out experiments on two classes of examples.

\subsection{Case 1}
Suppose the elliptic curve $A_{/F}$ has rank zero over $F$. Under our assumptions, we expect to generically have $r_\text{alg}(A/E)=0$ as well, and we can numerically check it for a given quadratic extension $E/F$ by adding the $F$-rank of $A$ and that of its twist $A^-$ with respect to $E/F$. When that is the case, Conjecture \ref{lowerrank} implies that the local root number $\varepsilon_{\frak{q}}$ determines whether the plectic point $\phi_{\mbox{\tiny $\mrm{Tate}$}}(\mrm{Q}_{A})$ is trivial or not. Our computations are collected in Table~\ref{table1} and support the conjecture. 

\noindent  Let $\bb{Q}_9=\bb{Q}_3(\sqrt{-1})$ be the unramified quadratic extension of $\bb{Q}_3$. The quantity $\log_S(\mrm{Q}_A)$, defined in Section \ref{Numerical evidence}, belongs to $\bb{Q}_9\otimes_{\bb{Q}_3}\bb{Q}_9$ and it is a multiple of the elementary tensor $\sqrt{-1}\otimes\sqrt{-1}$. In the last column of the table we report the scaling factor. Moreover, the column labeled ``ATR extension'' in the table reports the quantities $\beta$ such that $E=F(\sqrt{\beta})$.

\vspace{0.2cm}
\begin{table}[h]
  \begin{center}
  \begin{tabular}{llllr}
    \toprule
    $\bs{D_F}$ & \textbf{Curve} &$\bs{\varepsilon_{\frak{q}}}$& \textbf{ATR extension} & $\bs{\log_S(\mrm{Q}_A)}$\\ 
\midrule
    $13$ &  \lmfdb{13}{36} 1a2 &$+1$& $\hspace{1.9mm}-9w+8$&$2 \cdot 3^{2} + 2 \cdot 3^{3} + 2 \cdot 3^{4} + O(3^{5})$\\
          &                         &$+1$&$-12w+17$&$2 \cdot 3^{2} + 2 \cdot 3^{3} + 3^{4} + O( 3^{5})$\\
	 \noalign{\smallskip}\hline\noalign{\smallskip}
    $37$    & \lmfdb{37}{36} 1a2 &$+1$& $\hspace{1.9mm}-4w+9$& $2 \cdot 3^{2} + 3^{4} + O(3^{5})$ \\
	\noalign{\smallskip}\hline\noalign{\smallskip}
   $37$ & \lmfdb{37}{36} 1b1 && $\hspace{2mm}-4w+9,\ \hspace{1.9mm} -12w+29,\ -12w-7$ \\
&&$-1$& $-35w+17,\ -21w+62,\ -47w+29$ & $0+O(3^5)$\\
&&& $-39w+125$\\
\noalign{\smallskip}\hline\noalign{\smallskip}
   $37$    & \lmfdb{37}{36} 1c1 && $\hspace{2mm}-4w+9,\ \ \hspace{0.4mm} -12w-19,\ -12w+29$\\
&&$-1$& $-12w-7,\ \hspace{2mm} -35w+17,\ -21w+62$ &$0+O(3^5)$\\
&&& $-47w+29,\ -39w+125$\\
   \bottomrule
  \end{tabular}
  \caption{}
  \label{table1}
  	\end{center}
\end{table}
\vspace{-0.8cm}

\subsection{Case 2}
Suppose the elliptic curve $A_{/F}$ has rank one over $F$. Under our assumptions, we expect to often have  $r_\mrm{alg}(A/E)=2$. When that is the case, Conjecture \ref{Kolyvagin} implies that the quantity $\varrho_A(S)$ determines whether the plectic point $\phi_{\mbox{\tiny $\mrm{Tate}$}}(\mrm{Q}_{A})$ is trivial or not.   When $A_{/F}$ has both a $3$-adic prime of split and non-split multiplicative reduction, $\varrho_A(S)=2$ and the plectic point should be non-zero. Otherwise, $\varrho_A(S)=3$ and we should have $\phi_{\mbox{\tiny $\mrm{Tate}$}}(\mrm{Q}_{A})=0$. Our computations are collected in Table \ref{table2} and support the expectations.

\noindent The column labeled ``Difference'' in the table reports the difference $\varrho_A(S)-r_\mrm{alg}(A/E)$.

\vspace{0.2cm}
 \begin{table}[h]
   \center
   \begin{tabular}{lllcr}
    \toprule
    $\bs{D_F}$ & \textbf{Curve} & \textbf{ATR extension}& \textbf{Difference} & $\bs{\log_S(\mrm{Q}_A)}$ \\
    \midrule
     $13$ &  \lmfdb{13}{153} 2e2 & $\hspace{1.9mm}-9w+8$& $0$& $2 \cdot 3^{2} + 3^{4} + 3^{5} + O(3^{6})$\\
\noalign{\smallskip}\hline\noalign{\smallskip}
     $13$ &  \lmfdb{13}{207} 1c1 & $\hspace{1.9mm}-9w-4$& $0$& $3^{3} + 3^{4} + O(3^{5})$ \\
     && $\hspace{1.9mm}-9w+8$& $0$ &$2 \cdot 3^{2} + 2 \cdot 3^{3} + O(3^{5})$ \\
     \noalign{\smallskip}\hline\noalign{\smallskip}
      $13$ & \lmfdb{13}{225} 1b2 & $\hspace{2mm}-3w-1$ & $1$& $0+O(3^5)$\\
 	 && $-12w+17$ & $1$ &$0+O(3^5)$\\
	\noalign{\smallskip}\hline\noalign{\smallskip}
	  $37$ & \lmfdb{37}{63} 1a2 & $\hspace{1.9mm}-4w+9$& $1$ &$0+O(3^5)$\\
	 \noalign{\smallskip}\hline\noalign{\smallskip}
	   $37$ & \lmfdb{37}{63} 1b1, &$\hspace{1.9mm}-4w+9$& $1$ &$0+O(3^5)$\\
	\noalign{\smallskip}\hline\noalign{\smallskip}
     $37$ & \lmfdb{37}{63} 1d1 & $\hspace{1.9mm}-4w+9$& $0$ & $3^{2} + 3^{3} + 2 \cdot 3^{4} + O(3^{5})$\\
	 \noalign{\smallskip}\hline\noalign{\smallskip}
	 $37$& \lmfdb{37}{63} 2a1& $\hspace{1.9mm}-3w+5$& $1$ &$0+O(3^5)$\\
	 \noalign{\smallskip}\hline\noalign{\smallskip}
	  $37$& \lmfdb{37}{63} 2b1& $\hspace{4.9mm}3w+5$& $1$ &$0+O(3^5)$\\
	 \noalign{\smallskip}\hline\noalign{\smallskip}
     $37$ & \lmfdb{37}{63} 2d1 & $\hspace{1.9mm}-3w+5$& $0$ & $3^{2} + 3^{3} + O(3^{6})$\\
          && $-21w+62$& $0$&$2 \cdot 3^{2} + 3^{6} + O(3^7)$\\
          && $-32w+41$& $0$&$3^3 + 3^4 + O(3^7)$\\
		  \noalign{\smallskip}\hline\noalign{\smallskip}
     $37$ & \lmfdb{37}{99} 2c1& $-20w+29$& $0$&$3^{2} + 3^{3} + O(3^{4})$\\
          &&  $\hspace{1.9mm}-9w+14$& $0$&$3^{2} + 2 \cdot 3^{3} + O(3^{4})$\\
          && $-12w+29$& $0$&$3^{3} + O(3^4)$\\
          && $-32w+41$& $0$&$3^{3} + O(3^4)$\\
          && $-12w-7$ & $0$&$3^{2} + 2 \cdot 3^{3} + O(3^{4})$\\
          && $-35w+17$&$0$&$3^{2} + 2 \cdot 3^{3} + O(3^{4})$\\
    \bottomrule
   \end{tabular}
   \caption{}\label{table2}
\end{table}

\noindent The evidence for Conjecture \ref{algebraicity} is weaker  when plectic points are non-zero, given the low precision of our calculations. We plan to extend the methods of \cite{GM-Oc} to the setting of this paper to improve our results. Nevertheless, we computed the quantity $\log_S(\mrm{Q}_A)$ to seven $3$-adic digits of precision for two examples with encouraging results: one example is reported in Section \ref{Numerical evidence} of the Introduction, and the other is described next.
 We consider again the elliptic curve \lmfdb{37}{63} 2d1 defined over $F=\mathbb{Q}(\sqrt{37})$ of Weierstrass model
  \[
    A_{/F} \colon\ y^2 + x y + y = x^{3} + w x^{2} + \left(w + 1\right) x + 2, \qquad w = \frac{1+\sqrt{37}}2.
  \]
The elliptic curve has rank $2$ over the ATR extension $E=F\big(\sqrt{\beta}\big)$, for $\beta=-32w+41$, where all prime divisors of $\frak{f}_A$ are inert. Since $\varrho_A(S)=2$, Conjectures \ref{algebraicity} and \ref{Kolyvagin} imply that the plectic point should be non-zero and explicitly related to a generator of $\wedge^2A(E)$. We use \texttt{Magma} to compute the generators of $A(E)$: $P_1 = \big(2+w,\ -6-2w\big)$ and
  \[
     P_2 = \left(\frac{44074+ 33068w}{118943},\  \frac{13134267 - 83850648w}{2572261318} \sqrt{\beta} - \frac{163017+33068w}{237886}  \right).
  \]
Then, by setting $P_S=\det_S(P_1\wedge P_2)$, we compute that
  \[
  \log_S(\mrm{Q}_A)\equiv8\cdot \log_A(P_S) \pmod{3^7}.  
  \]

\bibliography{Plectic}

\begin{thebibliography}{{LMF}21}

\bibitem[BCP97]{magma}
W.~Bosma, J.~Cannon, and C.~Playoust.
\newblock The {M}agma algebra system. {I}. {T}he user language.
\newblock {\em J. Symbolic Comput.}, 24(3-4):235--265, 1997.
\newblock Computational algebra and number theory (London, 1993).

\bibitem[BD95]{derivedheights}
M.~Bertolini and H.~Darmon.
\newblock Derived {$p$}-adic heights.
\newblock {\em Amer. J. Math.}, 117(6):1517--1554, 1995.

\bibitem[BG18]{AnticyclotomicBG}
F.~Bergunde and L.~Gehrmann.
\newblock Leading terms of anticyclotomic stickelberger elements and $p$-adic
  periods.
\newblock {\em Trans. Amer. Math. Soc.}, 370(9), 2018.

\bibitem[CW77]{CoatesWiles}
J.~Coates and A.~Wiles.
\newblock On the conjecture of {B}irch and {S}winnerton-{D}yer.
\newblock {\em Invent. Math.}, 39(3):223--251, 1977.

\bibitem[Dar01]{IntegrationDarmon}
H.~Darmon.
\newblock Integration on $\cal{H}_p\times\cal{H}$ and arithmetic applications.
\newblock {\em Annals of Mathematics}, 154(3):589--639, 2001.

\bibitem[DR17]{DR2}
H.~Darmon and V.~Rotger.
\newblock Diagonal cycles and {E}uler systems {II}: {T}he {B}irch and
  {S}winnerton-{D}yer conjecture for {H}asse-{W}eil-{A}rtin {$L$}-functions.
\newblock {\em J. Amer. Math. Soc.}, 30(3):601--672, 2017.

\bibitem[FG21a]{plecticHeegner}
M.~Fornea and L.~Gehrmann.
\newblock Plectic {S}tark-{H}eegner points.
\newblock {\em Preprint}, 2021.

\bibitem[FG21b]{PlecticJacobians}
M.~Fornea and L.~Gehrmann.
\newblock {The arithmetic of plectic Jacobians}.
\newblock {\em In preparation}, 2021.

\bibitem[GM14]{GM-Oc}
X.~Guitart and M.~Masdeu.
\newblock Overconvergent cohomology and quaternionic {D}armon points.
\newblock {\em J. Lond. Math. Soc. (2)}, 90(2):495--524, 2014.

\bibitem[GMM20]{AutomorphicDarmon}
X.~Guitart, M.~Masdeu, and S.~Molina.
\newblock An automorphic approach to {D}armon points.
\newblock {\em Indiana Univ. Math. J.}, 69(4):1251--1274, 2020.

\bibitem[GMS15]{ArbitraryDarmon}
X.~Guitart, M.~Masdeu, and M.~H. Sengun.
\newblock Darmon points on elliptic curves over number fields of arbitrary
  signature.
\newblock {\em Proceedings of the London Mathmatic Society}, 11(2):484--518,
  2015.

\bibitem[Gol82]{GoldfeldBSD}
D.~Goldfeld.
\newblock Sur les produits partiels eul\'{e}riens attach\'{e}s aux courbes
  elliptiques.
\newblock {\em C. R. Acad. Sci. Paris S\'{e}r. I Math.}, 294(14):471--474,
  1982.

\bibitem[Gre09]{Greenberg}
M.~Greenberg.
\newblock Stark-heegner points and the cohomology of quaternionic shimura
  varieties.
\newblock {\em Duke Math. J.}, 2009.

\bibitem[GZ86]{GZformula}
B.~H. Gross and D.~B. Zagier.
\newblock Heegner points and derivatives of {$L$}-series.
\newblock {\em Invent. Math.}, 84(2):225--320, 1986.

\bibitem[Kol88]{Koly}
V.~A. Kolyvagin.
\newblock Finiteness of {E(Q)} and {S}h({E},{Q}) for a subclass of {W}eil
  curves.
\newblock {\em Izv. Akad. Nauk SSSR Ser. Mat.}, 52(3):522--540, 670--671, 1988.

\bibitem[{LMF}21]{lmfdb}
The {LMFDB Collaboration}.
\newblock The {L}-functions and modular forms database.
\newblock \url{http://www.lmfdb.org}, 2021.
\newblock [Online; accessed 14 April 2021].

\bibitem[Nek16]{NekRubinfest}
J.~Nekov\'ar.
\newblock {Some remarks on the BSD conjecture}.
\newblock {\em Rubinfest's talk at Harvard}, 2016.

\bibitem[NS16]{PlecticNS}
J.~Nekov\'{a}\v{r} and A.~J. Scholl.
\newblock Introduction to plectic cohomology.
\newblock In {\em Advances in the theory of automorphic forms and their
  {$L$}-functions}, volume 664 of {\em Contemp. Math.}, pages 321--337. Amer.
  Math. Soc., Providence, RI, 2016.

\bibitem[{Sag}20]{sagemath}
{Sage Developers}.
\newblock {\em {S}ageMath, the {S}age {M}athematics {S}oftware {S}ystem
  ({V}ersion 9.2)}, 2020.
\newblock {\tt http://www.sagemath.org}.

\bibitem[Sch85]{Schoof}
R.~Schoof.
\newblock Elliptic curves over finite fields and the computation of square
  roots mod {$p$}.
\newblock {\em Math. Comp.}, 44(170):483--494, 1985.

\bibitem[Ski20]{Skinner}
C.~Skinner.
\newblock A converse to a theorem of {G}ross, {Z}agier, and {K}olyvagin.
\newblock {\em Ann. of Math. (2)}, 191(2):329--354, 2020.

\bibitem[SU14]{IwasawaSU}
C.~Skinner and E.~Urban.
\newblock The {I}wasawa main conjectures for {$\text{GL}_2$}.
\newblock {\em Invent. Math.}, 195(1):1--277, 2014.

\bibitem[Var98]{Varshavsky}
Y.~Varshavsky.
\newblock {$p$}-adic uniformization of unitary {S}himura varieties. {II}.
\newblock {\em J. Differential Geom.}, 49(1):75--113, 1998.

\bibitem[Zha01]{Heights}
S.~Zhang.
\newblock Heights of {H}eegner points on {S}himura curves.
\newblock {\em Ann. of Math. (2)}, 153(1):27--147, 2001.

\end{thebibliography}
\bibliographystyle{alpha}

\end{document}